\documentclass{amsart}
\usepackage{amsmath}
\usepackage{amssymb}
\usepackage{enumerate}
\usepackage{mathrsfs}
\usepackage{color}
\usepackage{todonotes}
\usepackage{tikz}
\usepackage{algorithmic,algorithm}
\usepackage{multirow}
\usepackage{graphicx}
\usepackage{subfigure}
\usepackage{caption}
\usepackage{pgfplots}
\pgfplotsset{width=10cm,height=6cm,compat=1.9}
\definecolor{darkblue}{rgb}{0.0, 0.0, 0.55}
\definecolor{bordeaux}{rgb}{0.34, 0.01, 0.1}
\usepackage[pagebackref,colorlinks,linkcolor=bordeaux,citecolor=darkblue,urlcolor=black,hypertexnames=true]{hyperref}

\newtheorem{theorem}{Theorem}[section]
\newtheorem{lemma}[theorem]{Lemma}
\newtheorem{definition}[theorem]{Definition}
\newtheorem{example}[theorem]{Example}

\newtheorem{conjecture}[theorem]{Conjecture}

\newtheorem{remark}[theorem]{Remark}

\def\Q{{\mathbb{Q}}}
\def\R{{\mathbb{R}}}

\def\N{{\mathbb{N}}}

\def\x{{\mathbf{x}}}

\def\a{{\boldsymbol{\alpha}}}

\def\b{{\boldsymbol{\beta}}}
\def\ba{{\boldsymbol{a}}}
\def\bv{{\boldsymbol{v}}}
\def\bu{{\boldsymbol{u}}}
\def\bw{{\boldsymbol{w}}}
\def\A{{\mathscr{A}}}
\def\B{{\mathscr{B}}}

\def\supp{\hbox{\rm{supp}}}

\def\SONC{\hbox{\rm{SONC}}}

\def\int{\hbox{\rm{int}}}
\def\New{\hbox{\rm{New}}}
\def\Conv{\hbox{\rm{conv}}}

\begin{document}

\title[A second order cone characterization for SONC]{A second order cone characterization for sums of nonnegative circuits}
\thanks{The first author was supported by China Postdoctoral Science Foundation under grants 2018M641055. 
The second author was supported by the FMJH Program PGMO (EPICS project) and  EDF, Thales, Orange et Criteo. 
Both authors have benefited from the Tremplin ERC Stg Grant ANR-18-ERC2-0004-01 (T-COPS project). 
The second author's work has been supported by European Union’s Horizon 2020 research and innovation programme under the Marie Sklodowska-Curie Actions, grant agreement 813211 (POEMA) as well as from the AI Interdisciplinary Institute ANITI funding, through the French "Investing for the Future – PIA3" program under the Grant agreement n$^{\circ}$ANR-19-PI3A-0004}
\author{Jie Wang and Victor Magron}
\date{\today}

\begin{abstract}
The second-order cone is a class of simple convex cones and optimizing over them can be done more efficiently than with semidefinite programming. 
It is interesting both in theory and in practice to investigate which convex cones admit a representation using second-order cones, given that they have a strong expressive ability. 
In this paper, we prove constructively that the cone of sums of nonnegative circuits (SONC) admits a second-order cone representation. 
Based on this, we give a new algorithm to compute SONC decompositions for certain classes of nonnegative polynomials via second-order cone programming. 
Numerical experiments demonstrate the efficiency of our algorithm for polynomials with a fairly large size.
\end{abstract}

\maketitle
\bibliographystyle{amsplain}

\section{Introduction}
A {\em circuit polynomial} is of the form
$$\sum_{\a\in\A}c_{\a}\x^{\a}-d\x^{\b}\in\R[\x]=\R[x_1,\ldots,x_n],$$
where $c_{\a}>0$ for all $\a\in\A$, $\A\subseteq(2\N)^n$ comprises the vertices of a simplex and $\b$ lies in the interior of this simplex.
The set of {\em sums of nonnegative circuit polynomials (SONC)} was introduced by Iliman and Wolff in \cite{iw} as a new certificate of nonnegativity for sparse polynomials, which is independent of the well-known set of sums of squares (SOS). 
Another recently introduced alternative certificates \cite{ca} are sums of arithmetic-geometric-exponentials (SAGE), which can be obtained via relative entropy programming.
The connection between SONC and SAGE polynomials have been recently studied in \cite{mu,wang,ka19}. 
It happens  that both cones can be unified by a single one \cite{mu}, and that each cone element  has a cancellation-free representation in term of generators \cite{mu,wang}.
 
One of the significant differences between SONC and SOS is that SONC decompositions preserve the sparsity of polynomials while SOS decompositions do not in general \cite{wang}. 
The set of SONC polynomials with a given support forms a convex cone, called a {\em SONC cone}. Optimization problems over SONC cones can be formulated as geometric programs (see \cite{lw} for the unconstrained case and \cite{dlw,diw,dkw} for the constrained case). Numerical experiments for unconstrained polynomial optimization problems in \cite{se} have demonstrated the advantage of the SONC-based methods compared to the SOS-based methods, especially in the high-degree but fairly sparse case.

In the SOS case, there have been several attempts to exploit sparsity occurring in (un-)constrained polynomial optimization problems. 
The sparse variant~\cite{Waki06SparseSOS,Las06SparseSOS} of the moment-SOS hierarchy (also called Lasserre's hierarchy) exploits the correlative sparsity pattern between the input variables to reduce the support of the resulting SOS decompositions.
Such sparse representation results have been successfully applied in many fields, such as optimal power-flow~\cite{Josz16}, roundoff error bounds~\cite{ma17,ma18} and recently extended to the noncommutative case~\cite{ncsparse}.
Another way to exploit sparsity is to consider patterns based on monomial terms (rather than variables), yielding an alternative sparse variant of Lasserre's hierarchy~\cite{tssos}.

One of the similar features shared by SOS/SONC-based frameworks is their intrinsic connections with conic programming: SOS decompositions are computed via semidefinite programming and SONC decompositions via geometric programming. 
In both cases, the resulting optimization problems are solved with interior-point algorithms, thus output approximate nonnegativity certificates. 
However, one can still obtain an exact certificate from such output via hybrid numerical-symbolic algorithms when the input polynomial lies in the interior of the SOS/SONC cone.
One way is to rely on rounding-projection algorithms adapted to the SOS cone \cite{pe} and the SONC cone \cite{ma19}, or alternatively on  perturbation-compensation schemes \cite{univsos,multivsos18} available within the {\tt RealCertify} \cite{RealCertify} library.

In this paper, we study the second-order cone representation of SONC cones. An $n$-dimensional {\em (rotated) second-order cone} is defined as
\begin{equation*}
\mathbf{K}^n:=\{\ba\in\R^{n}\mid2a_1a_2\ge\sum_{i=3}^{n}a_i^2,a_1\ge0,a_2\ge0\}.
\end{equation*}
The second-order cone is well-studied and has mature solvers. Optimizing via second-order cone programming (SOCP) can be handled more efficiently than with semidefinite programming \cite{ah,al}. On the other hand, despite the simplicity of second-order cones, they have a strong ability to express other convex cones (many such examples can be found in \cite[Section 3.3]{ben}). Therefore, it is interesting in theory and also important from the  applications point of view to investigate which convex cones can be expressed by second-order cones.

Given sets of lattice points $\A\subseteq(2\N)^n$, $\B_1\subseteq\Conv(\A)\cap(2\N)^n$ and $\B_2\subseteq\Conv(\A)\cap(\N^n\backslash(2\N)^n)$ ($\Conv(\A)$ is the convex hull of $\A$) with $\A\cap\B_1=\varnothing$, the SONC cone supported on $\A,\B_1,\B_2$ is defined as
\begin{align*}
\SONC_{\A,\B_1,\B_2}:=&\{(\mathbf{c}_{\A},\mathbf{d}_{\B_1},\mathbf{d}_{\B_2})\in\R_+^{|\A|}\times\R_+^{|\B_1|}\times\R^{|\B_2|}\\
&\mid\sum_{\a\in\A} c_{\a}\x^{\a}-\sum_{\b\in\B_1\cup\B_2}d_{\b}\x^{\b}\in\SONC\},
\end{align*}
where $\mathbf{c}_{\A}=(c_{\a})_{\a\in\A}$, $\mathbf{d}_{\B_1}=(d_{\b})_{\b\in\B_1}$ and $\mathbf{d}_{\B_2}=(d_{\b})_{\b\in\B_2}$. The first main result of this paper is the following theorem.
\begin{theorem}\label{mthm}
For $\A\subseteq(2\N)^n$, $\B_1\subseteq\Conv(\A)\cap(2\N)^n$ and $\B_2\subseteq\Conv(\A)\cap(\N^n\backslash(2\N)^n)$ with $\A\cap\B_1=\varnothing$, the convex cone $\SONC_{\A,\B_1,\B_2}$ admits a second-order cone representation.
\end{theorem}

The proof of Theorem \ref{mthm} is constructive and involves writing a SONC polynomial as a sum of binomial squares with rational exponents (Theorem \ref{sec3-thm5}). This enables us to propose a new algorithm, based on SOCP, providing SONC decompositions for a particular class of nonnegative polynomials, which in turn yields lower bounds for unconstrained polynomial optimization problems.
We test the algorithm on various randomly generated polynomials up to a fairly large size, involving  $n \sim 40$ variables and of degree $d \sim 60$. The numerical results demonstrate the efficiency of our algorithm.

The rest of this paper is organized as follows. In Section 2, we list some preliminaries on SONC polynomials. In Section 3, we reveal a key connection between SONC polynomials and sums of binomial squares by introducing the notion of $\A$-rational mediated sets. By virtue of this connection, we obtain second-order cone representations for SONC cones in Section 4. In Section 5, we provide an algorithm which outputs SONC decompositions for certain classes of nonnegative polynomials via SOCP. Numerical experiments are provided in Section 6.

\section{Preliminaries}
Let $\R[\x]=\R[x_1,\ldots,x_n]$ be the ring of real $n$-variate polynomial, and let $\R_+$ be the set of positive real numbers. For a finite set $\A\subseteq\N^n$, we denote by $\Conv(\A)$ the convex hull of $\A$. Given a finite set $\A\subseteq\N^n$, we consider polynomials $f\in\R[\x]$ supported on $\A\subseteq\N^n$, i.e., $f$ is of the form $f(\x)=\sum_{\a\in \A}c_{\a}\x^{\a}$ with $c_{\a}\in\R, \x^{\a}=x_1^{\alpha_1}\cdots x_n^{\alpha_n}$. The support of $f$ is $\supp(f):=\{\a\in \A\mid c_{\a}\ne0\}$ and the Newton polytope of $f$ is defined as $\New(f):=\Conv(\supp(f))$. For a polytope $P$, we use $V(P)$ to denote the vertex set of $P$ and use $P^{\circ}$ to denote the interior of $P$. For a set $A$, we use $\#A$ to denote the cardinality of $A$.

A polynomial $f\in\R[\x]$ is nonnegative over $\R^n$ is called a {\em nonnegative polynomial}, or a {\em positive semi-definite (PSD) polynomial}.

The following definition of circuit polynomials was proposed by Iliman and De Wolff in \cite{iw}.
\begin{definition}
A polynomial $f\in\R[\x]$ is called a {\em circuit polynomial} if it is of the form
\begin{equation}\label{nc-eq}
f(\x)=\sum_{\a\in\A}c_{\a}\x^{\a}-d\x^{\b},
\end{equation}
and satisfies the following conditions:
\begin{enumerate}[(i)]
    \item $\A\subseteq(2\N)^n$ comprises the vertices of a simplex;
    \item $c_{\a}>0$ for each $\a\in\A$;
    \item $\b\in\Conv(\A)^{\circ}\cap\N^n$.
\end{enumerate}
\end{definition}

For a circuit polynomial $f=\sum_{\a\in\A}c_{\a}\x^{\a}-d\x^{\b}$, from the definition we can uniquely write
\begin{equation}
\b=\sum_{\a\in\A}\lambda_{\a}\a\textrm{ with } \lambda_{\a}>0 \textrm{ and } \sum_{\a\in\A}\lambda_{\a}=1.
\end{equation}
Then we define the corresponding {\em circuit number} as $\Theta_f:=\prod_{\a\in\A}(c_{\a}/\lambda_{\a})^{\lambda_{\a}}$. The nonnegativity of the circuit polynomial $f$ is decided by its circuit number alone, that is, $f$ is nonnegative if and only if either $\b\notin(2\N)^n$ and $|d|\le\Theta_f$, or $\b\in(2\N)^n$ and $d\le\Theta_f$ (\cite[Theorem 3.8]{iw}).

\begin{remark}
To provide a concise narrative, we also view a monomial square as a nonnegative circuit polynomial.
\end{remark}

An explicit representation of a polynomial being a {\em sum of nonnegative circuit polynomials}, or {\em SONC} for short, provides a certificate for its nonnegativity. Such a certificate is called a {\em SONC decomposition}. For simplicity, we also denote the set of SONC polynomials by SONC.

For a polynomial $f\in\R[\x]$, let
$$\Lambda(f):=\{\a\in\supp(f)\mid\a\in(2\N)^n\textrm{ and }c_{\a}>0\}$$
and $\Gamma(f):=\supp(f)\backslash\Lambda(f)$. Then we can write $f$ as
$$f=\sum_{\a\in\Lambda(f)}c_{\a}\x^{\a}-\sum_{\b\in\Gamma(f)}d_{\b}\x^{\b}.$$

For each $\b\in\Gamma(f)$, let
\begin{equation}\label{fb}
\mathscr{F}(\b):=\{\Delta\mid\Delta\textrm{ is a simplex, } \b\in\Delta^{\circ}, V(\Delta)\subseteq\Lambda(f)\}.
\end{equation}
In \cite[Theorem 5.5]{wang}, it was proved that if $f\in\SONC$, then $f$ admits a SONC decomposition
\begin{equation}\label{sec2-eq1}
f=\sum_{\b\in\Gamma(f)}\sum_{\Delta\in\mathscr{F}(\b)}f_{\b\Delta}+\sum_{\a\in\tilde{\A}}c_{\a}\x^{\a},
\end{equation}
where $f_{\b\Delta}$ is a nonnegative circuit polynomial supported on $V(\Delta)\cup\{\b\}$ for each $\Delta$ and $\tilde{\A}=\{\a\in\Lambda(f)\mid\a\notin\cup_{\b\in\Gamma(f)}\cup_{\Delta\in\mathscr{F}(\b)}V(\Delta)\}$.

\section{SONC polynomials and sums of binomial squares}
In this section, we give a characterization of SONC polynomials in terms of sums of binomial squares with rational exponents.

\subsection{Rational mediated sets}
\label{sec:ratmediatedsets}
A lattice point $\a\in\N^n$ is {\em even} if it is in $(2\N)^n$. For a subset $M\subseteq\N^n$, define
$$\overline{A}(M):=\{\frac{1}{2}(\bv+\bw)\mid\bv\ne\bw,\bv,\bw\in M\cap(2\N)^n\}$$
as the set of averages of distinct even points in $M$. A subset $\A\subseteq(2\N)^n$ is called a {\em trellis} if $\A$ comprises the vertices of a simplex. For a trellis $\A$, we say that $M$ is an {\em $\A$-mediated set} if $\A\subseteq M\subseteq\overline{A}(M)\cup\A$ (\cite{re}).
\begin{theorem}\label{sec3-thm1}
Let $f=\sum_{\a\in\A} c_{\a}\x^{\a}-d\x^{\b}\in\R[\x]$ with $d\ne0$ be a nonnegative circuit polynomial. Then $f$ is a sum of binomial squares if and only if there exists an $\A$-mediated set containing $\b$. Moreover, suppose that $\b$ belongs to an $\A$-mediated set $M$ and for each $\bu\in M\backslash\A$, let us write $\bu=\frac{1}{2}(\bv_\bu+\bw_\bu)$ for some $\bv_\bu \ne\bw_\bu \in M\cap(2\N)^n$. 
Then one obtains the decomposition  $f=\sum_{\bu\in M\backslash\A}(a_{\bu}\x^{\frac{1}{2}\bv_\bu}-b_{\bu}\x^{\frac{1}{2}\bw_\bu})^2$, with $a_{\bu},b_{\bu}\in\R$.
\end{theorem}
\begin{proof}
For the proof, the reader may refer to Theorem 5.2 in \cite{iw}.
\end{proof}

By Theorem \ref{sec3-thm1}, if we want to represent a nonnegative circuit polynomial as a sum of binomial squares, we need to first decide if there exists an $\A$-mediated set containing a given lattice point and then to compute one if there exists. However, there are obstacles for each of these two steps: (1) there may not exist such an $\A$-mediated set containing a given lattice point; (2) even if such a set exists, there is no efficient algorithm to compute it. In order to overcome these two difficulties, we introduce the concept of $\A$-rational mediated sets as a replacement of $\A$-mediated sets by admitting rational numbers in coordinates. 

Concretely, for a subset $M\subseteq\Q^n$, let us define
$$\widetilde{A}(M):=\{\frac{1}{2}(\bv+\bw)\mid\bv\ne\bw,\bv,\bw\in M\}$$
as the set of averages of distinct rational points in $M$. Let us assume that $\A\subseteq\Q^n$ comprises the vertices of a simplex. We say that $M$ is an {\em $\A$-rational mediated set} if $\A\subseteq M\subseteq\widetilde{A}(M)\cup\A$. We shall see that for a trellis $\A$ and a lattice point $\b\in\Conv(\A)^{\circ}$, an $\A$-rational mediated set containing $\b$ always exists and moreover, there is an effective algorithm to compute it.

First, let us consider the one dimensional case. For a sequence of integer numbers $A=\{s,q_1,\ldots,q_m,p\}$ (arranged from small to large), if every $q_i$ is an average of two distinct numbers in $A$, then we say $A$ is an {\em $(s,p)$-mediated sequence}. Note that the property of $(s,p)$-mediated sequences is preserved under translations, that is, there is a one-to-one correspondence between $(s,p)$-mediated sequences and $(s+r,p+r)$-mediated sequences for any integer number $r$. So it suffices to consider the case of $s=0$.

For a fixed $p$ and an integer $0<q<p$, a {\em minimal $(0,p)$-mediated sequence} containing $q$ is a $(0,p)$-mediated sequence containing $q$ with the least number of elements. 

\begin{example}
Consider the set $A=\{0,2,4,5,8,11\}$. One can easily check by hand that $A$ is a minimal $(0,11)$-mediated sequence containing $2,4,5,8$.
\end{example}

Denote the number of elements in a minimal $(0,p)$-mediated sequence containing $q$ by $N(\frac{q}{p})$. One can then easily show that $N(\frac{1}{p})=\left\lceil\log_2(p)\right\rceil+2$ by induction on $p$. We conjecture that this formula holds for general $q$, i.e.,
\begin{conjecture}
If gcd$(p,q) = 1$, then $N(\frac{q}{p})=\left\lceil\log_2(p)\right\rceil+2$.
\end{conjecture}

Generally we do not know how to compute a minimal $(0,p)$-mediated sequence containing a given $q$. However, we have an algorithm to compute an approximately minimal $(0,p)$-mediated sequence containing a given $q$ as the following lemma shows.
\begin{lemma}\label{sec4-lm3}
For $p,q\in\N,0<q<p$, there exists a $(0,p)$-mediated sequence containing $q$ with the number of elements less than $\frac{1}{2}(\log_2(p)+\frac{3}{2})^2$.
\end{lemma}
\begin{proof}
We can assume gcd$(p,q) = 1$ (otherwise consider $p/(p,q),q/(p,q)$ instead). Let us do induction on $p$. Assume that for any $p',q'\in\N,0<q'<p'<p$, there exists a $(0,p')$-mediated sequence containing $q'$ with the number of elements less than $\frac{1}{2}(\log_2(p')+\frac{3}{2})^2$.

{\bf Case 1:} Suppose that $p$ is an even number. If $q=\frac{p}{2}$, then by gcd$(p,q) = 1$, we have $q=1$ and $A=\{0,1,2\}$ is a $(0,p)$-mediated sequence containing $q$. Otherwise, we have either $0<q<\frac{p}{2}$ or $\frac{p}{2}<q<p$. For $0<q<\frac{p}{2}$, by the induction hypothesis, there exists a $(0,\frac{p}{2})$-mediated sequence $A'$ containing $q$. For $\frac{p}{2}<q<p$, since the property of mediated sequences is preserved under translations, one can first subtract $\frac{p}{2}$ and obtain a $(0,\frac{p}{2})$-mediated sequence containing $q-\frac{p}{2}$ by the induction hypothesis. Then by adding $\frac{p}{2}$, one obtain a $(\frac{p}{2},p)$-mediated sequence $A'$ containing $q$. It follows that $A=A'\cup\{p\}$ or $A=\{0\}\cup A'$ is a $(0,p)$-mediated sequence containing $q$.
\begin{center}
\begin{tikzpicture}
\draw (0,0)--(6,0);
\fill (0,0) circle (1pt);
\node[above] (1) at (0,0) {$0$};
\fill (2,0) circle (1pt);
\node[above] (2) at (2,0) {$q$};
\fill (3,0) circle (1pt);
\node[above] (6) at (3,0) {$\frac{p}{2}$};
\fill (4,0) circle (1pt);
\node[above] (3) at (4,0) {$(q)$};
\fill (6,0) circle (1pt);
\node[above] (7) at (6,0) {$p$};
\end{tikzpicture}
\end{center}
Moreover, we have
$$\#A=1+\#A'<1+\frac{1}{2}(\log_2(\frac{p}{2})+\frac{3}{2})^2<\frac{1}{2}(\log_2(p)+\frac{3}{2})^2.$$

{\bf Case 2:} Suppose that $p$ is an odd number. Without loss of generality, assume that $q$ is an even number (otherwise one can consider $p-q$ instead and then obtain a $(0,p)$-mediated sequence containing $q$ through the map $x\mapsto p-x$ which clearly preserves the property of mediated sequences).

Let $q=2^kr$ for some $k,r\in\N\backslash\{0\}$ and $2\nmid r$. If $q=p-r$, then $q=\frac{q-r+p}{2}$. Since gcd$(p,q) = 1$, we have $r=1$. Let $A=\{0,\frac{1}{2}q,\frac{3}{4}q,\ldots,(1-\frac{1}{2^{k}})q,q,p\}$. For $1\le i\le k$, we have $(1-\frac{1}{2^{i}})q=\frac{1}{2}(1-\frac{1}{2^{i-1}})q+\frac{1}{2}q$. Therefore, $A$ is a $(0,p)$-mediated sequence containing $q$.
\begin{center}
\begin{tikzpicture}
\draw (0,0)--(11,0);
\fill (0,0) circle (1pt);
\node[above] (1) at (0,0) {$0$};
\fill (5,0) circle (1pt);
\node[above] (2) at (5,0) {$\frac{1}{2}q$};
\node[above] (3) at (8,0) {$\cdots$};
\fill (7.5,0) circle (1pt);
\node[above] (4) at (7.5,0) {$\frac{3}{4}q$};
\fill (9,0) circle (1pt);
\node[above] (5) at (9,0) {$(1-\frac{1}{2^{k}})q$};
\node[below] (8) at (9,0) {$q-r$};
\fill (10,0) circle (1pt);
\node[above] (6) at (10,0) {$q$};
\fill (11,0) circle (1pt);
\node[above] (7) at (11,0) {$p$};
\end{tikzpicture}
\end{center}
Moreover, we have
$$\#A=k+3<\frac{1}{2}(\log_2(2^k+1)+\frac{3}{2})^2=\frac{1}{2}(\log_2(p)+\frac{3}{2})^2.$$

If $q<p-r$, then $q$ lies on the line segment between $q-r$ and $\frac{q-r+p}{2}$. Since $\frac{q-r+p}{2}-(q-r)=\frac{p+r-q}{2}<p$, then by the induction hypothesis, there exists a $(q-r,\frac{q-r+p}{2})$-mediated sequence $A'$ containing $q$ (using translations). It follows that $A=\{0,\frac{1}{2}q,\frac{3}{4}q,\ldots,(1-\frac{1}{2^{k-1}})q,p\}\cup A'$ is a $(0,p)$-mediated sequence containing $q$.
\begin{center}
\begin{tikzpicture}
\draw (0,0)--(11,0);
\fill (0,0) circle (1pt);
\node[above] (1) at (0,0) {$0$};
\fill (4.3,0) circle (1pt);
\node[above] (2) at (4.3,0) {$\frac{1}{2}q$};
\node[above] (3) at (6.9,0) {$\cdots$};
\fill (6.45,0) circle (1pt);
\node[above] (4) at (6.45,0) {$\frac{3}{4}q$};
\fill (7.8,0) circle (1pt);
\node[above] (5) at (7.8,0) {$(1-\frac{1}{2^{k}})q$};
\node[below] (8) at (7.8,0) {$q-r$};
\fill (8.6,0) circle (1pt);
\node[above] (6) at (8.6,0) {$q$};
\fill (9.4,0) circle (1pt);
\node[above] (8) at (9.4,0) {$\frac{q-r+p}{2}$};
\fill (11,0) circle (1pt);
\node[above] (7) at (11,0) {$p$};
\end{tikzpicture}
\end{center}
Moreover, we have
\begin{align*}
\#A&=k+1+\#A'\\
&<\log_2(\frac{q}{r})+1+\frac{1}{2}(\log_2(\frac{p+r-q}{2})+\frac{3}{2})^2\\
&<\log_2(p)+1+\frac{1}{2}(\log_2(\frac{p}{2})+\frac{3}{2})^2\\
&=\frac{1}{2}(\log_2(p)+\frac{3}{2})^2.
\end{align*}

If $q>p-r$, then $q$ lies on the line segment between $\frac{q-r+p}{2}$ and $p$. Since $p-\frac{q-r+p}{2}=\frac{p+r-q}{2}<p$, then by the induction hypothesis, there exists a $(\frac{q-r+p}{2},p)$-mediated sequence $A'$ containing $q$ (using translations). It follows that the set $A=\{0,\frac{1}{2}q,\frac{3}{4}q,\ldots,(1-\frac{1}{2^{k}})q\}\cup A'$ is a $(0,p)$-mediated sequence containing $q$.
\begin{center}
\begin{tikzpicture}
\draw (0,0)--(11,0);
\fill (0,0) circle (1pt);
\node[above] (1) at (0,0) {$0$};
\fill (4.7,0) circle (1pt);
\node[above] (2) at (4.7,0) {$\frac{1}{2}q$};
\node[above] (3) at (7.7,0) {$\cdots$};
\fill (7.2,0) circle (1pt);
\node[above] (4) at (7.2,0) {$\frac{3}{4}q$};
\fill (8.6,0) circle (1pt);
\node[above] (5) at (8.6,0) {$(1-\frac{1}{2^{k}})q$};
\node[below] (8) at (8.6,0) {$q-r$};
\fill (9.8,0) circle (1pt);
\node[above] (6) at (9.8,0) {$\frac{q-r+p}{2}$};
\fill (10.3,0) circle (1pt);
\node[above] (8) at (10.3,0) {$q$};
\fill (11,0) circle (1pt);
\node[above] (7) at (11,0) {$p$};
\end{tikzpicture}
\end{center}
As previously, we have $\#A =k+1+\#A' \leq \frac{1}{2}(\log_2(p)+\frac{3}{2})^2 $.
\end{proof}


\begin{lemma}\label{sec4-lm2}
Suppose that $\a_1$ and $\a_2$ are two rational points, and $\b$ is any rational point on the line segment between $\a_1$ and $\a_2$. Then there exists an $\{\a_1,\a_{2}\}$-rational mediated set $M$ containing $\b$. Furthermore, if the denominators of coordinates of $\a_1,\a_2,\b$ are odd numbers, and the numerators of coordinates of $\a_1,\a_2$ are even numbers, then we can ensure that the denominators of coordinates of points in $M$ are odd numbers and the numerators of coordinates of points in $M\backslash\{\b\}$ are even numbers.
\end{lemma}
\begin{proof}
Suppose $\b=(1-\frac{q}{p})\a_1+\frac{q}{p}\a_2$, $p,q\in\N,0<q<p$,gcd$(p,q) = 1$. We then construct a one-to-one correspondence between the points on the one-dimensional number axis and the points on the line across $\a_1$ and $\a_2$ via the map:
$$s\mapsto(1-\frac{s}{p})\a_1+\frac{s}{p}\a_2,$$
such that $\a_1$ corresponds to the origin, $\a_2$ corresponds to $p$ and $\b$ corresponds to $q$. 
Then it is clear that a $(0,p)$-mediated sequence containing $q$ corresponds to a $\{\a_1,\a_{2}\}$-rational mediated set containing $\b$.
Hence by Lemma \ref{sec4-lm3}, there exists a $\{\a_1,\a_{2}\}$-rational mediated set $M$ containing $\b$ with the number of elements less than $\frac{1}{2}(\log_2(p)+\frac{3}{2})^2$. Moreover, we can see that if $\a_1,\a_2,\b$ are lattice points, then the elements in $M$ are also lattice points.

If the denominators of coordinates of $\a_1,\a_2,\b$ are odd numbers, and the numerators of coordinates of $\a_1,\a_2$ are even numbers, assume that the least common multiple of denominators appearing in the coordinates of $\a_1,\a_2,\b$ is $r$ and then remove the denominators by multiplying the coordinates of $\a_1,\a_2,\b$ by $r$ such that $r\a_1,r\a_2$ are even lattice points. If $r\b$ is even, let $M'$ be the $\{\frac{r}{2}\a_1,\frac{r}{2}\a_{2}\}$-rational mediated set containing $\frac{r}{2}\b$ obtained as above (note that the elements in $M'$ are lattice points). Then $M=\frac{2}{r}M':=\{\frac{2}{r}\bu\mid\bu\in M'\}$ is an $\{\a_1,\a_{2}\}$-rational mediated set containing $\b$ such that the denominators of coordinates of points in $M$ are odd numbers and the numerators of coordinates of points in $M\backslash\{\b\}$ are even numbers as desired.

If $r\b$ is not even, assume without loss of generality that $\b$ lies on the line segment between $\a_1$ and $\frac{\a_1+\a_2}{2}$. Let $\b'=2\b-\a_1$ with $r\b'$ an even lattice point. Let $M'$ be the $\{\frac{r}{2}\a_1,\frac{r}{2}\a_{2}\}$-rational mediated set containing $\frac{r}{2}\b'$ obtained as above (note that the elements in $M'$ are lattice points). Then $M=\frac{2}{r}M'\cup\{\b\}$ is an $\{\a_1,\a_{2}\}$-rational mediated set containing $\b$ such that the denominators of coordinates of points in $M$ are odd numbers and the numerators of coordinates of points in $M\backslash\{\b\}$ are even numbers as desired.
\end{proof}

\begin{lemma}\label{sec4-lm5}
For a trellis $\A=\{\a_1,\ldots,\a_m\}$ and a lattice point $\b\in\Conv(\A)^{\circ}$, there exists an $\A$-rational mediated set $M_{\A\b}$ containing $\b$ such that the denominators of coordinates of points in $M_{\A\b}$ are odd numbers and the numerators of coordinates of points in $M_{\A\b}\backslash\{\b\}$ are even numbers.
\end{lemma}
\begin{proof}
Suppose $\b=\sum_{i=1}^m\frac{q_i}{p}\a_i$, where $p=\sum_{i=1}^mq_i$, $p,q_i\in\N\backslash\{0\}$, $(p,q_1,\ldots,q_m)=1$. If $p$ is an even number, then because $(p,q_1,\ldots,q_m)=1$, there must exist an odd number among the $q_i$'s. Without loss of generality assume $q_1$ is an odd number. If $p$ is an odd number and there exists an even number among the $q_i$'s, then without loss of generality assume $q_1$ is an even number. In any of these two cases, we have
$$\b=\frac{q_1}{p}\a_1+\frac{p-q_1}{p}(\frac{q_2}{p-q_1}\a_2+\cdots+\frac{q_m}{p-q_1}\a_m).$$
Let $\b_{1}=\frac{q_2}{p-q_1}\a_2+\cdots+\frac{q_m}{p-q_1}\a_m$. Then $\b=\frac{q_1}{p}\a_1+\frac{p-q_1}{p}\b_{1}$.

If $p$ is an odd number and all $q_i$'s are odd numbers, then we have
\begin{align*}
\b&=\frac{q_1}{q_1+q_2}(\frac{q_1+q_2}{p}\a_1+\frac{q_3}{p}\a_3+\cdots+\frac{q_m}{p}\a_m)\\
&+\frac{q_2}{q_1+q_2}(\frac{q_1+q_2}{p}\a_2+\frac{q_3}{p}\a_3+\cdots+\frac{q_m}{p}\a_m).
\end{align*}
Let $\b_{1}=\frac{q_1+q_2}{p}\a_1+\frac{q_3}{p}\a_3+\cdots+\frac{q_m}{p}\a_m$ and $\b_{2}=\frac{q_1+q_2}{p}\a_2+\frac{q_3}{p}\a_3+\cdots+\frac{q_m}{p}\a_m$. Then $\b=\frac{q_1}{q_1+q_2}\b_{1}+\frac{q_2}{q_1+q_2}\b_{2}$.

Apply the same procedure for $\b_1$ (and $\b_2$), and continue iteratively. Eventually we obtain a set of points $\{\b_i\}_{i=1}^l$ such that for each $i$, $\b_i=\lambda_i\b_j+\mu_i\b_k$ or $\b_i=\lambda_i\b_j+\mu_i\a_k$ or $\b_i=\lambda_i\a_j+\mu_i\a_k$, where $\lambda_i+\mu_i=1,\lambda_i,\mu_i>0$. We claim that the denominators of coordinates of any $\b_i$ are odd numbers, and the numerators of coordinates of any $\b_i$ are even numbers. This is because for each $\b_i$, we have the expression $\b_i=\sum_{j}\frac{s_j}{r}\a_j$, where $r$ is an odd number and all $\a_j$'s are even lattice points. For $\b_i=\lambda\b_j+\mu\b_k$ (or $\b_i=\lambda\b_j+\mu\a_k$, $\b_i=\lambda\a_j+\mu\a_k$ respectively), let $M_i$ be the $\{\b_{j},\b_{k}\}$- (or $\{\b_{j},\a_{k}\}$-, $\{\a_{j},\a_{k}\}$- respectively) rational mediated set containing $\b_i$ obtained by Lemma \ref{sec4-lm2} satisfying that the denominators of coordinates of points in $M_i$ are odd numbers and the numerators of coordinates of points in $M_i\backslash\{\b\}$ are even numbers for $i=0,\ldots,l$ (set $\b_0=\b$). Let $M_{\A\b}=\cup_{i=0}^lM_i$. Then $M_{\A\b}$ is clearly an $\A$-rational mediated set containing $\b$ with the desired property.
\end{proof}

\subsection{Representing SONC polynomials as sums of binomial squares}
\label{sec:soncbinomial}
For $r\in\N$ and $f(\x)\in\R[\x]$, let $f(\x^r):=f(x_1^r,\ldots,x_n^r)$. For any odd $r\in\N$, it is clear that $f(\x)=\sum_{\a\in\A}c_{\a}\x^{\a}-d\x^{\b}$ is a nonnegative circuit polynomial if and only if $f(\x^r)=\sum_{\a\in\A}c_{\a}\x^{r\a}-d\x^{r\b}$ is a nonnegative circuit polynomial.
\begin{theorem}\label{sec3-thm3}
Let $f=\sum_{\a\in\A} c_{\a}\x^{\a}-d\x^{\b}\in\R[\x]$ with $d\ne0$ be a circuit polynomial. Assume that $M_{\A\b}$ is the $\A$-rational mediated set containing $\b$ provided by Lemma \ref{sec4-lm5}.
and for each $\bu\in M_{\A\b}\backslash\A$, let $\bu=\frac{1}{2}(\bv_{\bu}+\bw_{\bu}), \bv_{\bu}\ne\bw_{\bu}\in M_{\A\b}$. Then $f$ is nonnegative if and only if $f$ can be written as $f=\sum_{\bu\in M_{\A\b}\backslash\A}(a_{\bu}\x^{\frac{1}{2}\bv_{\bu}}-b_{\bu}\x^{\frac{1}{2}\bw_{\bu}})^2$, $a_{\bu},b_{\bu}\in\R$.
\end{theorem}
\begin{proof}
Assume that the least common multiple of denominators appearing in the coordinates of points in $M_{\A\b}$ is $r$, which is odd. Then $f(\x)$ is nonnegative if and only if $f(\x^r)$ is nonnegative. Multiply all coordinates of points in $M_{\A\b}$ by $r$ to remove the denominators, and the obtained $rM_{\A\b}:=\{r\bu\mid\bu\in M_{\A\b}\}$ is an $r\A$-mediated set containing $r\b$. Hence by Theorem \ref{sec3-thm1}, $f(\x^r)$ is nonnegative if and only if $f(\x^r)$ can be written as $f(\x^r)=\sum_{\bu\in M_{\A\b}\backslash \A}(a_{\bu}\x^{\frac{r}{2}\bv_{\bu}}-b_{\bu}\x^{\frac{r}{2}\bw_{\bu}})^2$, $a_{\bu},b_{\bu}\in\R$, which is equivalent to $f(\x)=\sum_{\bu\in M_{\A\b}\backslash\A}(a_{\bu}\x^{\frac{1}{2}\bv_{\bu}}-b_{\bu}\x^{\frac{1}{2}\bw_{\bu}})^2$.
\end{proof}

\begin{example}
Let $f=x^4y^2+x^2y^4+1-3x^2y^2$ be the Motzkin's polynomial and $\A=\{\a_1=(0,0),\a_2=(4,2),\a_3=(2,4)\}$, $\b=(2,2)$. Then $\b=\frac{1}{3}\a_1+\frac{1}{3}\a_2+\frac{1}{3}\a_3=\frac{1}{2}(\frac{1}{3}\a_1+\frac{2}{3}\a_2)+\frac{1}{2}(\frac{1}{3}\a_1+\frac{2}{3}\a_3)$. Let $\b_1=\frac{1}{3}\a_1+\frac{2}{3}\a_2$ and $\b_2=\frac{1}{3}\a_1+\frac{2}{3}\a_3$ such that $\b=\frac{1}{2}\b_1+\frac{1}{2}\b_2$. Let $\b_3=\frac{2}{3}\a_1+\frac{1}{3}\a_2$ and $\b_4=\frac{2}{3}\a_1+\frac{1}{3}\a_3$. Then it is easy to check that $M=\{\a_1,\a_2,\a_3,\b,\b_1,\b_2,\b_3,\b_4\}$ is an $\A$-rational mediated set containing $\b$.

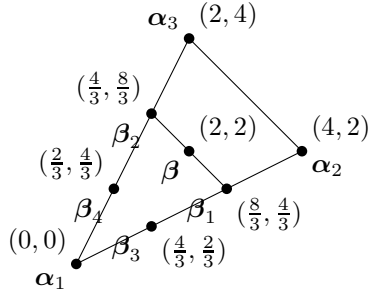
\begin{center}
\begin{tikzpicture}
\draw (0,0)--(1.5,3);
\draw (0,0)--(3,1.5);
\draw (1.5,3)--(3,1.5);
\draw (1,2)--(2,1);
\fill (0,0) circle (2pt);
\node[above left] (1) at (0,0) {$(0,0)$};
\node[below left] (1) at (0,0) {$\a_1$};
\fill (1.5,3) circle (2pt);
\node[above right] (2) at (1.5,3) {$(2,4)$};
\node[above left] (2) at (1.5,3) {$\a_3$};
\fill (3,1.5) circle (2pt);
\node[above right] (3) at (3,1.5) {$(4,2)$};
\node[below right] (3) at (3,1.5) {$\a_2$};
\fill (1.5,1.5) circle (2pt);
\node[above right] (4) at (1.5,1.5) {$(2,2)$};
\node[below left] (4) at (1.5,1.5) {$\b$};
\fill (1,2) circle (2pt);
\node[above left] (5) at (1,2) {$(\frac{4}{3},\frac{8}{3})$};
\node[below left] (5) at (1,2) {$\b_2$};
\fill (2,1) circle (2pt);
\node[below right] (6) at (2,1) {$(\frac{8}{3},\frac{4}{3})$};
\node[below left] (6) at (2,1) {$\b_1$};
\fill (0.5,1) circle (2pt);
\node[above left] (5) at (0.5,1) {$(\frac{2}{3},\frac{4}{3})$};
\node[below left] (5) at (0.5,1) {$\b_4$};
\fill (1,0.5) circle (2pt);
\node[below right] (6) at (1,0.5) {$(\frac{4}{3},\frac{2}{3})$};
\node[below left] (6) at (1,0.5) {$\b_3$};
\end{tikzpicture}
\end{center}

By Theorem \ref{sec3-thm3}, assume that $f=x^4y^2+x^2y^4+1-3x^2y^2=(a_1x^{\frac{2}{3}}y^{\frac{4}{3}}-b_1x^{\frac{4}{3}}y^{\frac{2}{3}})^2+(a_2xy^2-b_2x^{\frac{1}{3}}y^{\frac{2}{3}})^2+(a_3x^{\frac{2}{3}}y^{\frac{4}{3}}-b_3)^2
+(a_4x^2y-b_4x^{\frac{2}{3}}y^{\frac{1}{3}})^2+(a_5x^{\frac{4}{3}}y^{\frac{2}{3}}-b_5)^2$.
By comparing the coefficients, one obtains $f=\frac{3}{2}(x^{\frac{2}{3}}y^{\frac{4}{3}}-x^{\frac{4}{3}}y^{\frac{2}{3}})^2+(xy^2-x^{\frac{1}{3}}y^{\frac{2}{3}})^2+\frac{1}{2}(x^{\frac{2}{3}}y^{\frac{4}{3}}-1)^2
+(x^2y-x^{\frac{2}{3}}y^{\frac{1}{3}})^2+\frac{1}{2}(x^{\frac{4}{3}}y^{\frac{2}{3}}-1)^2$. 
Here we represent $f$ as a sum of five binomial squares with rational exponents.
\end{example}

\begin{lemma}\label{sec3-lm6}
Let $f(\x)\in\R[\x]$. For an odd number $r$, $f(\x)\in\SONC$ if and only if $f(\x^r)\in\SONC$.
\end{lemma}
\begin{proof}
It easily follows from the fact that $f(\x)$ is a nonnegative circuit polynomial if and only if $f(\x^r)$ is a nonnegative circuit polynomial for an odd number $r$.
\end{proof}

\begin{theorem}\label{sec3-thm5}
Let $f=\sum_{\a\in\Lambda(f)}c_{\a}\x^{\a}-\sum_{\b\in\Gamma(f)}d_{\b}\x^{\b}\in\R[\x]$. Let $\mathscr{F}(\b)$ be as in \eqref{fb}. For every $\b\in\Gamma(f)$ and every $\Delta\in\mathscr{F}(\b)$, let $M_{\b\Delta}$ be the $V(\Delta)$-rational mediated set containing $\b$ provided by Lemma \ref{sec4-lm5}.
Let $M=\cup_{\b\in\Gamma(f)}\cup_{\Delta\in\mathscr{F}(\b)}M_{\b\Delta}$. For each $\bu\in M\backslash\Lambda(f)$, let $\bu=\frac{1}{2}(\bv_{\bu}+\bw_{\bu}),\bv_{\bu}\ne\bw_{\bu}\in M$. Let $\tilde{\A}=\{\a\in\Lambda(f)\mid\a\notin\cup_{\b\in\Gamma(f)}\cup_{\Delta\in\mathscr{F}(\b)}V(\Delta)\}$. Then $f\in\SONC$ if and only if $f$ can be written as $f=\sum_{\bu\in M\backslash\Lambda(f)}(a_{\bu}\x^{\frac{1}{2}\bv_{\bu}}-b_{\bu}\x^{\frac{1}{2}\bw_{\bu}})^2+\sum_{\a\in\tilde{\A}}c_{\a}\x^{\a}$, $a_{\bu},b_{\bu}\in\R$.
\end{theorem}
\begin{proof}
Suppose $f\in\SONC$. By Theorem 5.5 in \cite{wang}, we can write $f$ as $f=\sum_{\b\in\Gamma(f)}\sum_{\Delta\in\mathscr{F}(\b)}f_{\b\Delta}+\sum_{\a\in\tilde{\A}}c_{\a}\x^{\a}$ such that every $f_{\b\Delta}=\sum_{\a\in V(\Delta)}c_{\b\Delta\a}\x^{\a}-d_{\b\Delta}\x^{\b}$ is a nonnegative circuit polynomial. We have $f_{\b\Delta}=\sum_{\bu\in M_{\A\b}\backslash\A}(a_{\bu}\x^{\frac{1}{2}\bv_{\bu}}-b_{\bu}\x^{\frac{1}{2}\bw_{\bu}})^2$, $a_{\bu},b_{\bu}\in\R$ by Theorem \ref{sec3-thm3}. Thus $f=\sum_{\bu\in M\backslash\Lambda(f)}(a_{\bu}\x^{\frac{1}{2}\bv_{\bu}}-b_{\bu}\x^{\frac{1}{2}\bw_{\bu}})^2+\sum_{\a\in\tilde{\A}}c_{\a}\x^{\a}$, $a_{\bu},b_{\bu}\in\R$.

Suppose $f=\sum_{\bu\in M\backslash\Lambda(f)}(a_{\bu}\x^{\frac{1}{2}\bv_{\bu}}-b_{\bu}\x^{\frac{1}{2}\bw_{\bu}})^2+\sum_{\a\in\tilde{\A}}c_{\a}\x^{\a}$, $a_{\bu},b_{\bu}\in\R$. Assume that the least common multiple of denominators appearing the coordinates of points in $M$ is $r$, which is odd. Then $f(\x^r)=\sum_{\bu\in M\backslash\Lambda(f)}(a_{\bu}\x^{\frac{r}{2}\bv_{\bu}}-b_{\bu}\x^{\frac{r}{2}\bw_{\bu}})^2+\sum_{\a\in\tilde{\A}}c_{\a}\x^{r\a}$, $a_{\bu},b_{\bu}\in\R$, which is a sum of nonnegative circuit polynomials since every binomial square (and monomial square) is a nonnegative circuit polynomial. Hence by Lemma \ref{sec3-lm6}, $f(\x)\in\SONC$.
\end{proof}

\section{Second-order cone representations of SONC cones}

Second-order cone programming (SOCP) plays an important role in convex optimization and can be handled via very efficient algorithms. It is meaningful to investigate which convex cone admits a second-order cone representation. If such a representation exists for a given convex cone, then it is possible to design efficient algorithms for optimization problems over the convex cone. In \cite{fa}, Fawzi proved that positive semidefinite cones do not admit any second-order cone representations in general, which implies that SOS cones do not admit any second-order cone representations in general. In this section, we prove that dramatically unlike the SOS cones, SONC cones always admit second-order cone representations. Let us discuss it in more details. We denote by $\mathcal{Q}^k:=\mathcal{Q}\times\cdots \times \mathcal{Q}$ the Cartesian product of $k$ copies of a second-order cone $\mathcal{Q}$. A {\em linear slice} of $\mathcal{Q}^k$ is an intersection of $\mathcal{Q}^k$ with a linear subspace.
\begin{definition}
A convex cone $C\subseteq\R^m$ has a {\em second-order cone lift of size $k$} (or simply a {\em $\mathcal{Q}^k$-lift}) if it can be written as the projection of a slice of $\mathcal{Q}^k$, that is, there is a subspace $L$ of $\mathcal{Q}^k$ and a linear map $\pi\colon\mathcal{Q}^k\rightarrow\R^m$ such that $C=\pi(\mathcal{Q}^k\cap L)$.
\end{definition}

We give the following definition of SONC cones supported on given lattice points.
\begin{definition}
Given sets of lattice points $\A\subseteq(2\N)^n$, $\B_1\subseteq\Conv(\A)\cap(2\N)^n$ and $\B_2\subseteq\Conv(\A)\cap(\N^n\backslash(2\N)^n)$ such that $\A\cap\B_1=\varnothing$, define the SONC cone supported on $\A,\B_1,\B_2$ as
\begin{align*}
\SONC_{\A,\B_1,\B_2}:=& \{(\mathbf{c}_{\A},\mathbf{d}_{\B_1},\mathbf{d}_{\B_2})\in\R_+^{|\A|}\times\R_+^{|\B_1|}\times\R^{|\B_2|}\\
&\mid\sum_{\a\in\A} c_{\a}\x^{\a}-\sum_{\b\in\B_1\cup\B_2}d_{\b}\x^{\b}\in\SONC \},
\end{align*}
where $\mathbf{c}_{\A}=(c_{\a})_{\a\in\A}$, $\mathbf{d}_{\B_1}=(d_{\b})_{\b\in\B_1}$ and $\mathbf{d}_{\B_2}=(d_{\b})_{\b\in\B_2}$. It is easy to check that $\SONC_{\A,\B_1,\B_2}$ is indeed a convex cone.
\end{definition}

Let $\mathbb{S}^2_+$ be the convex cone of $2\times2$ positive semidefinite matrices, i.e.,
\[\mathbb{S}^2_+:=\left\{\begin{bmatrix}a&b\\b&c\end{bmatrix}\in\R^{2\times2}\mid\begin{bmatrix}a&b\\b&c\end{bmatrix}\textrm{ is positive semidefinite}\right\}.\]
\begin{lemma}
$\mathbb{S}^2_+$ is a $3$-dimensional rotated second-order cone.
\end{lemma}
\begin{proof}
Suppose $A=\begin{bmatrix}a&b\\b&c\end{bmatrix}$ is a $2\times2$ symmetric matrix. The condition of $A$ to be positive semidefinite is $a\ge0,c\ge0,ac\ge b^2$. 
Thus $\mathbb{S}^2_+$ is a rotated second-order cone by definition.
\end{proof}

\begin{theorem}
For $\A\subseteq(2\N)^n$, $\B_1\subseteq\Conv(\A)\cap(2\N)^n$ and $\B_2\subseteq\Conv(\A)\cap(\N^n\backslash(2\N)^n)$ such that $\A\cap\B_1=\varnothing$, the convex cone $\SONC_{\A,\B_1,\B_2}$ has an $(\mathbb{S}^2_+)^k$-lift for some $k\in\N$.
\end{theorem}
\begin{proof}
For every $\b\in\B_1\cup\B_2$, let $\mathscr{F}(\b)$ be as in \eqref{fb}. Then
for every $\b\in\B_1\cup\B_2$ and every $\Delta\in\mathscr{F}(\b)$, let $M_{\b\Delta}$ be the $V(\Delta)$-rational mediated set containing $\b$ provided by Lemma \ref{sec4-lm5}.
Let $M=\cup_{\b\in\B_1\cup\B_2}\cup_{\Delta\in\mathscr{F}(\b)}M_{\b\Delta}$. 
For each $\bu_i\in M\backslash\A$, let us write  $\bu_i=\frac{1}{2}(\bv_i+\bw_i)$. Let $B=\cup_{\bu_i\in M\backslash\A}\{\frac{1}{2}\bv_i,\frac{1}{2}\bw_i\}$, $\tilde{\A}=\{\a\in\Lambda(f)\mid\a\notin\cup_{\b\in\Gamma(f)}\cup_{\Delta\in\mathscr{F}(\b)}V(\Delta)\}$ and $k=\#M\backslash\A+\#\tilde{\A}$.

Then by Theorem \ref{sec3-thm5}, a polynomial $f$ is in $\SONC_{\A,\B_1,\B_2}$ if and only if $f$ can be written as $f=\sum_{\bu_i\in M\backslash\A}(a_{i}\x^{\frac{1}{2}\bv_i}-b_{i}\x^{\frac{1}{2}\bw_i})^2+\sum_{\a\in\tilde{\A}}c_{\a}\x^{\a}$, $a_i,b_i\in\R$, which is equivalent to the existence of a symmetric matrix $Q=\sum_{i=1}^kQ_i$ such that $f=(\x^B)^TQ\x^B$, where $Q_i$ is a symmetric matrix with zeros everywhere except either at the four positions corresponding to the monomials $\x^{\frac{1}{2}\bv_i}$, $\x^{\frac{1}{2}\bw_i}$ or at the position corresponding to a monomial $\x^{\frac{1}{2}\a}$ for some  $\a\in\tilde{\A}$.
This leads respectively to either four entries forming a $2\times2$ positive semidefinite submatrix or one single  positive entry.

Let $\pi:(\mathbb{S}^2_+)^{k}\rightarrow\SONC_{\A,\B_1,\B_2}$
be the linear map that maps an element in $Q_1\times\cdots\times Q_k$ to the coefficient vector of $f$ which is in $\SONC_{\A,\B_1,\B_2}$ via the equality $f=(\x^B)^TQ\x^B$ with $Q=\sum_{i=1}^kQ_i$. So we obtain an $(\mathbb{S}^2_+)^k$-lift for $\SONC_{\A,\B_1,\B_2}$.

\end{proof}


\section{SONC optimization via second-order cone programming}
In this section, we tackle the unconstrained polynomial optimization problem via SOCP, based on the representation of SONC cones derived in the previous section.

The unconstrained polynomial optimization problem can be formulated as follows:
\begin{equation}\label{sec3-eq00}
(\textrm{P}):\quad\begin{cases}
\sup&\xi\\
\textrm{s.t.}&f(\x)-\xi\ge0,\quad\x\in\R^n.
\end{cases}
\end{equation}
Let us  denote by $\xi^*$ the optimal value of (\ref{sec3-eq00}).

Replace the nonnegativity constraint in (\ref{sec3-eq00}) by the following one to obtain a SONC relaxation of  problem (\ref{sec3-eq00}):
\begin{equation}\label{sec3-eq5}
(\textrm{SONC}):\quad\begin{cases}
\sup&\xi\\
\textrm{s.t.}&f(\x)-\xi\in\SONC.
\end{cases}
\end{equation}
Let us denote by $\xi_{sonc}$ the optimal value of (\ref{sec3-eq5}).

\subsection{Conversion to PN-polynomials}
Let $f=\sum_{\a\in\Lambda(f)}c_{\a}\x^{\a}-\sum_{\b\in\Gamma(f)}d_{\b}\x^{\b}\in\R[\x]$. If $d_{\b}>0$ for all $\b\in\Gamma(f)$, then we call $f$ a {\em PN-polynomial}. The ``PN" in PN-polynomial is short for ``positive part plus negative part". The positive part is given by $\sum_{\a\in\Lambda(f)}c_{\a}\x^{\a}$ and the negative part is given by $-\sum_{\b\in\Gamma(f)}d_{\b}\x^{\b}$. For a PN-polynomial $f(\x)$, it is clear that
\begin{equation*}
f(\x)\ge0\textrm{ for all }\x\in\R^n\Longleftrightarrow f(\x)\ge0\textrm{ for all }\x\in\R_+^n.
\end{equation*}
Moreover, we have
\begin{lemma}\label{sec4-lm}
Let $f(\x)\in\R[\x]$ be a PN-polynomial. Then for any positive integer $k$, $f(\x)\in\SONC$ if and only if $f(\x^k)\in\SONC$.
\end{lemma}
\begin{proof}
It is immediate from the fact that a polynomial $f(\x)$ with exactly one negative term is a nonnegative circuit polynomial if and only if $f(\x^k)$ is a nonnegative circuit polynomial for any positive integer $k\in\N$.
\end{proof}

\begin{theorem}\label{sec4-thm}
Let $f=\sum_{\a\in\Lambda(f)}c_{\a}\x^{\a}-\sum_{\b\in\Gamma(f)}d_{\b}\x^{\b}\in\R[\x]$ be a PN-polynomial. Let $\mathscr{F}(\b)$ be as in \eqref{fb}. For every $\b\in\Gamma(f)$ and every $\Delta\in\mathscr{F}(\b)$, let $M_{\b\Delta}$ be a $V(\Delta)$-rational mediated set containing $\b$. Let $M=\cup_{\b\in\Gamma(f)}\cup_{\Delta\in\mathscr{F}(\b)}M_{\b\Delta}$ and $\tilde{\A}=\{\a\in\Lambda(f)\mid\a\notin\cup_{\b\in\Gamma(f)}\cup_{\Delta\in\mathscr{F}(\b)}V(\Delta)\}$. For each $\bu\in M\backslash\Lambda(f)$, let $\bu=\frac{1}{2}(\bv+\bw)$. Then $f\in\SONC$ if and only if $f$ can be written as $f=\sum_{\bu\in M\backslash\Lambda(f)}(a_{\bu}\x^{\frac{1}{2}\bv}-b_{\bu}\x^{\frac{1}{2}\bw})^2+\sum_{\a\in\tilde{\A}}c_{\a}\x^{\a}$, $a_{\bu},b_{\bu}\in\R$.
\end{theorem}
\begin{proof}
It follows easily from Lemma \ref{sec4-lm} and Theorem \ref{sec3-thm1}.
\end{proof}

The significant difference between Theorem \ref{sec3-thm5} and Theorem \ref{sec4-thm} is that to represent a SONC PN-polynomial as a sum of binomial squares, we do not require the denominators of coordinates of points in $\A$-rational mediated sets to be odd. 
By virtue of this fact, for given trellis $\A=\{\a_1,\ldots,\a_m\}$ and lattice point $\b\in\Conv(\A)^{\circ}$, we can then construct an $\A$-rational mediated set $M_{\A\b}$ containing $\b$ which is smaller than that the one from Lemma \ref{sec4-lm5}.
\begin{lemma}\label{sec4-lm4}
For a trellis $\A=\{\a_1,\ldots,\a_m\}$ and a lattice point $\b\in\Conv(\A)^{\circ}$, there exists an $\A$-rational mediated set $M_{\A\b}$ containing $\b$.
\end{lemma}
\begin{proof}
Suppose that $\b=\sum_{i=1}^m\frac{q_i}{p}\a_i$, where $p=\sum_{i=1}^mq_i$, $p,q_i\in\N^*$, $(p,q_1,\ldots,q_m)=1$. We can write $$\b=\frac{q_1}{p}\a_1+\frac{p-q_1}{p}(\frac{q_2}{p-q_1}\a_2+\cdots+\frac{q_m}{p-q_1}\a_m).$$
Let $\b_{1}=\frac{q_2}{p-q_1}\a_2+\cdots+\frac{q_m}{p-q_1}\a_m$. Then $\b=\frac{q_1}{p}\a_1+\frac{p-q_1}{p}\b_{1}$. Apply the same procedure for $\b_1$, and continue like this. Eventually we obtain a set of points $\{\b_i\}_{i=0}^{m-2}$ (set $\b_0=\b$) such that $\b_i=\lambda_i\a_{i+1}+\mu_i\b_{i+1}$, $i=0,\ldots,m-3$ and $\b_{m-2}=\lambda_{m-2}\a_{m-1}+\mu_{m-2}\a_m$, where $\lambda_i+\mu_i=1,\lambda_i,\mu_i>0$, $i=0,\ldots,m-2$. For $\b_i=\lambda_i\a_{i+1}+\mu_i\b_{i+1}$ (resp. $\b_{m-2}=\lambda_{m-2}\a_{m-1}+\mu_{m-2}\a_m$), let $M_i$ be the $\{\a_{i+1},\b_{i+1}\}$- (resp. $\{\a_{m-1},\a_{m}\}$-) rational mediated set containing $\b_i$ obtained by Lemma \ref{sec4-lm2}, $i=0,\ldots,m-2$. Let $M_{\A\b}=\cup_{i=0}^{m-2}M_i$. Then clearly $M_{\A\b}$ is an $\A$-rational mediated set containing $\b$.
\end{proof}

\begin{example}\label{ex2}
Let $f=x^4y^2+x^2y^4+1-3x^2y^2$ be the Motzkin's polynomial and $\A=\{\a_1=(4,2),\a_2=(2,4),\a_3=(0,0)\}$, $\b=(2,2)$. Then $\b=\frac{1}{3}\a_1+\frac{1}{3}\a_2+\frac{1}{3}\a_3=\frac{1}{3}\a_1+\frac{2}{3}(\frac{1}{2}\a_2+\frac{1}{2}\a_3)$. Let $\b_1=\frac{1}{2}\a_2+\frac{1}{2}\a_3$ such that  $\b=\frac{1}{3}\a_1+\frac{2}{3}\b_1$. Let $\b_2=\frac{2}{3}\a_1+\frac{1}{3}\b_1$. Then it is easy to check that $M=\{\a_1,\a_2,\a_3,\b,\b_1,\b_2\}$ is an $\A$-rational mediated set containing $\b$.
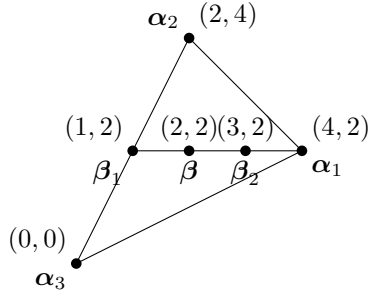
\begin{center}
\begin{tikzpicture}
\draw (0,0)--(1.5,3);
\draw (0,0)--(3,1.5);
\draw (1.5,3)--(3,1.5);
\draw (0.75,1.5)--(3,1.5);
\fill (0,0) circle (2pt);
\node[above left] (1) at (0,0) {$(0,0)$};
\node[below left] (1) at (0,0) {$\a_3$};
\fill (1.5,3) circle (2pt);
\node[above right] (2) at (1.5,3) {$(2,4)$};
\node[above left] (2) at (1.5,3) {$\a_2$};
\fill (3,1.5) circle (2pt);
\node[above right] (3) at (3,1.5) {$(4,2)$};
\node[below right] (3) at (3,1.5) {$\a_1$};
\fill (1.5,1.5) circle (2pt);
\node[above] (4) at (1.5,1.5) {$(2,2)$};
\node[below] (4) at (1.5,1.5) {$\b$};
\fill (0.75,1.5) circle (2pt);
\node[above left] (5) at (0.75,1.5) {$(1,2)$};
\node[below left] (5) at (0.75,1.5) {$\b_1$};
\fill (2.25,1.5) circle (2pt);
\node[above] (6) at (2.25,1.5) {$(3,2)$};
\node[below] (6) at (2.25,1.5) {$\b_2$};
\end{tikzpicture}
\end{center}
By a simple computation, we have $f=(1-xy^2)^2+2(x^{\frac{1}{2}}y-x^{\frac{3}{2}}y)^2+(xy-x^2y)^2$. Here we represent $f$ as a sum of three binomial squares with rational exponents.
\end{example}

For a polynomial $f=\sum_{\a\in\Lambda(f)}c_{\a}\x^{\a}-\sum_{\b\in\Gamma(f)}d_{\b}\x^{\b}\in\R[\x]$, we associate $f$ with the PN-polynomial $\tilde{f}=\sum_{\a\in\Lambda(f)}c_{\a}\x^{\a}-\sum_{\b\in\Gamma(f)}|d_{\b}|\x^{\b}$.
\begin{lemma}\label{sec4-lm1}
Suppose $f=\sum_{\a\in\Lambda(f)}c_{\a}\x^{\a}-\sum_{\b\in\Gamma(f)}d_{\b}\x^{\b}\in\R[\x]$. If $\tilde{f}$ is nonnegative, then $f$ is nonnegative. Moreover, $\tilde{f}\in\SONC$ if and only if $f\in\SONC$.
\end{lemma}
\begin{proof}
For any $\x\in\R^n$, we have
\begin{align*}
f(\x)&=\sum_{\a\in\Lambda(f)}c_{\a}\x^{\a}-\sum_{\b\in\Gamma(f)}d_{\b}\x^{\b}\\
&\ge\sum_{\a\in\Lambda(f)}c_{\a}|\x|^{\a}-\sum_{\b\in\Gamma(f)}|d_{\b}||\x|^{\b}\\
&=\tilde{f}(|\x|),
\end{align*}
where $|\x|=(|x_1|,\ldots,|x_n|)$. It follows that the nonnegativity of $\tilde{f}$ implies the nonnegativity of $f$.

For every $\b\in\Gamma(f)$, let $\mathscr{F}(\b)$ be as in \eqref{fb}. Let $\B=\{\b\in\Gamma(f)\mid\b\notin(2\N)^n\textrm{ and }d_{\b}<0\}$ and  $\tilde{\A}=\{\a\in\Lambda(f)\mid\a\notin\cup_{\b\in\Gamma(f)}\cup_{\Delta\in\mathscr{F}(\b)}V(\Delta)\}$. Assume $\tilde{f}\in\SONC$. Then we can write
\begin{align*}
\tilde{f}=&\sum_{\b\in\Gamma(f)\backslash\B}\sum_{\Delta\in\mathscr{F}(\b)}(\sum_{\a\in V(\Delta)}c_{\b\Delta\a}\x^{\a}-d_{\b\Delta}\x^{\b})\\
&+\sum_{\b\in\B}\sum_{\Delta\in\mathscr{F}(\b)}(\sum_{\a\in V(\Delta)}c_{\b\Delta\a}\x^{\a}-\tilde{d}_{\b\Delta}\x^{\b})+\sum_{\a\in\tilde{\A}}c_{\a}\x^{\a}
\end{align*}
such that each $\sum_{\a\in V(\Delta)}c_{\b\Delta\a}\x^{\a}-d_{\b\Delta}\x^{\b}$ and each $\sum_{\a\in V(\Delta)}c_{\b\Delta\a}\x^{\a}-\tilde{d}_{\b\Delta}\x^{\b}$ are nonnegative circuit polynomials. Note that $\sum_{\a\in V(\Delta)}c_{\b\Delta\a}\x^{\a}+\tilde{d}_{\b\Delta}\x^{\b}$ is also a nonnegative circuit polynomial and $\sum_{\Delta\in\Delta(\b)}\tilde{d}_{\b\Delta}=|d_{\b}|=-d_{\b}$ for any $\b\in\B$. Hence,
\begin{align*}
f=&\sum_{\b\in\Gamma(f)\backslash\B}\sum_{\Delta\in\mathscr{F}(\b)}(\sum_{\a\in V(\Delta)}c_{\b\Delta\a}\x^{\a}-d_{\b\Delta}\x^{\b})\\
&+\sum_{\b\in\B}\sum_{\Delta\in\mathscr{F}(\b)}(\sum_{\a\in V(\Delta)}c_{\b\Delta\a}\x^{\a}+\tilde{d}_{\b\Delta}\x^{\b})+\sum_{\a\in\tilde{\A}}c_{\a}\x^{\a}\in\SONC.
\end{align*}

The inverse follows similarly.
\end{proof}

Hence by Lemma \ref{sec4-lm1}, if we replace the polynomial $f$ in (\ref{sec3-eq5}) by its associated PN-polynomial $\tilde{f}$, then this does not affect the optimal value of (\ref{sec3-eq5}):
\begin{equation}\label{sec3-eq6}
(\textrm{SONC-PN}):\quad\begin{cases}
\sup&\xi\\
\textrm{s.t.}&\tilde{f}(\x)-\xi\in\SONC.
\end{cases}
\end{equation}

\begin{remark}
Lemma \ref{sec4-lm1} actually tells us that the SONC formulation for the polynomial optimization problem \eqref{sec3-eq5} always provides the optimal value of the corresponding PN-polynomial. This may greatly affect the quality of the lower bounds obtained via SONC decompositions, as we shall see in Sec.~\ref{quality}.
\end{remark}

\subsection{Compute a simplex cover}
Given a polynomial $f=\sum_{\a\in\Lambda(f)}c_{\a}\x^{\a}-\sum_{\b\in\Gamma(f)}d_{\b}\x^{\b}\in\R[\x]$, in order to obtain a SONC decomposition of $f$, we use all simplices containing $\b$ for each $\b\in\Gamma(f)$ in Theorem \ref{sec3-thm5}. In practice, we do not need that many simplices. Actually, by Carath\'eodory's theorem (\cite[Corollary 17.1.2]{rock}), we can write every SONC polynomial $f$ as a sum of at most $\#\supp(f)$
nonnegative circuit polynomials.

\begin{example}
Let $f=50x^4y^4+x^4+3y^4+800-100xy^2-100x^2y$. Let $\a_1=(0,0),\a_2=(4,0),\a_3=(0,4),\a_4=(4,4)$ and $\b_1=(2,1),\b_2=(1,2)$. 
There are two simplices which cover $\b_1$: one with vertices $\{\a_1,\a_2\,\a_3\}$, denoted by $\Delta_1$, and one with vertices $\{\a_1,\a_2\,\a_4\}$, denoted by $\Delta_2$. 
There are two simplices which cover $\b_2$: $\Delta_1$ and one with vertices $\{\a_1,\a_3\,\a_4\}$, denoted by $\Delta_3$. One can check that $f$ admits a SONC decomposition $f=g_1+g_2$, where $g_1=20x^4y^4+x^4+400-100x^2y$, supported on $\Delta_2$, and $g_2=30x^4y^4+3y^4+400-100xy^2$, supported on $\Delta_3$, are both nonnegative circuit polynomials. Hence the simplex $\Delta_1$ is not needed in this SONC decomposition of $f$.
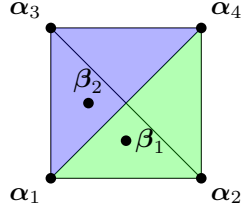
\begin{center}
\begin{tikzpicture}
\draw (0,0)--(0,2);
\draw (0,0)--(2,0);
\draw (2,0)--(2,2);
\draw (0,2)--(2,2);
\draw (0,0)--(2,2);
\draw (0,2)--(2,0);
\fill[fill=green,fill opacity=0.3] (0,0)--(2,0)--(2,2)--(0,0);
\fill[fill=blue,fill opacity=0.3] (0,0)--(0,2)--(2,2)--(0,0);
\fill (0,0) circle (2pt);
\node[below left] (1) at (0,0) {$\a_1$};
\fill (2,0) circle (2pt);
\node[below right] (2) at (2,0) {$\a_2$};
\fill (0,2) circle (2pt);
\node[above left] (3) at (0,2) {$\a_3$};
\fill (2,2) circle (2pt);
\node[above right] (4) at (2,2) {$\a_4$};
\fill (1,0.5) circle (2pt);
\node[right] (5) at (1,0.5) {$\b_1$};
\fill (0.5,1) circle (2pt);
\node[above] (6) at (0.5,1) {$\b_2$};
\end{tikzpicture}
\end{center}
\end{example}

Therefore we need to compute a set of simplices with vertices coming from $\Lambda(f)$ and that covers $\Gamma(f)$.  

For $\b\in\Gamma(f)$ and $\a_0\in\Lambda(f)$, define the following auxiliary linear program:
\begin{alignat*}{3}
&\textrm{SimSel}(\b,\Lambda(f),\a_0)=\,&\textrm{Argmax}&\quad&&\lambda_{\a_0}\\
&&\textrm{s.t.}&&&\sum_{\a\in\Lambda(f)}\lambda_{\a}\cdot\a=\b\\
&&&&&\sum_{\a\in\Lambda(f)}\lambda_{\a}=1\\
&&&&&\lambda_{\a}\ge0, \forall\a\in\Lambda(f).
\end{alignat*}

Following \cite{se}, we can ensure that the output of $\textrm{SimSel}(\b,\Lambda(f),\a_0)$ corresponds to a trellis which contains $\a_0$ and covers $\b$. 
We rely on the algorithm {\tt SimplexCover} to compute a simplex cover.

\bigskip

Let $\mathbf{K}$ be the $3$-dimensional rotated second-order cone, i.e.,
\begin{equation}
\mathbf{K}:=\{(a,b,c)\in\R^3\mid2ab\ge c^2,a\ge0,b\ge0\}.
\end{equation}

Suppose $\tilde{f}=\sum_{\a\in\Lambda(f)}c_{\a}\x^{\a}-\sum_{\b\in\Gamma(f)}d_{\b}\x^{\b}\in\R[\x]$. 
By Algorithm {\tt SimplexCover}, we compute a simplex cover $\{(\A_k,\b_k)\}_{k=1}^l$. 
For each $k$, let $M_{k}$ be an $\A_{k}$-rational mediated set containing $\b_k$ and $s_k=\#M_k\backslash\A_{k}$. 
For each $\bu_i^k\in M_k\backslash\A_{k}$, let us write  $\bu_i^k=\frac{1}{2}(\bv_i^k+\bw_i^k)$. 
Let $\tilde{\A}=\{\a\in\Lambda(f)\mid\a\notin\cup_{\b\in\Gamma(f)}\cup_{\Delta\in\mathscr{F}(\b)}V(\Delta)\}$. 
Then we can relax (SONC-PN) to an SOCP problem (SONC-SOCP) as follows:
\begin{equation}\label{sec3-eq7}
\begin{cases}
\sup&\xi\\
\textrm{s.t.}&\tilde{f}(\x)-\xi=\sum_{k=1}^l\sum_{i=1}^{s_k}(2a_i^k\x^{\bv_i^k}+b_i^k\x^{\bw_i^k}-2c_i^k\x^{\bu_i^k})+\sum_{\a\in\tilde{\A}}c_{\a}\x^{\a},\\
&(a_i^k,b_i^k,c_i^k)\in\mathbf{K},\quad\forall i,k.
\end{cases}
\end{equation}
Let us denote by $\xi_{socp}$ the optimal value of (\ref{sec3-eq7}). Then, we have $\xi_{socp}\le\xi_{sonc}\le\xi^*$.

\if{
\begin{algorithm}
\renewcommand{\algorithmicrequire}{\textbf{Input:}}
\renewcommand{\algorithmicensure}{\textbf{Output:}}
\newcommand{\roundfun}[2]{\texttt{round}(#1,#2)}
\caption{${\tt RoundProject}(f)$}\label{roundproject}
\begin{algorithmic}[1]
\REQUIRE
A SONC polynomial $f=\sum_{\a\in\Lambda(f)}c_{\a}\x^{\a}-\sum_{\b\in\Gamma(f)}d_{\b}\x^{\b}$, a rounding precision $\delta_i \in \N$, a precision parameter $\delta$ for the SOCP solver
\ENSURE
An exact SONC decomposition of $f$
\STATE $\{(\A_k,\b_k)\}_{k=1}^l:={\tt SimplexCover}(\Lambda(f),\Gamma(f))$;
\FOR {$k=1:l$}
\STATE $\{(\bu_i^k,\bv_i^k,\bw_i^k)\}_{i=1}^{s_k}:={\tt MedSet}(\A_k,\b_k)$;
\ENDFOR
\STATE $B:=\cup_{k=1}^l\cup_{i=1}^{s_k}\{\bv_i^k,\bw_i^k\}$;
\STATE Compute the matrix $\tilde Q$ by solving (\textrm{SONC-SOCP}) with $\tilde Q_i^k=\begin{pmatrix}2\tilde{a}_i^k&-\tilde{c}_i^k\\-\tilde{c}_i^k&\tilde{b}_i^k\end{pmatrix}$ at precision $\delta$;
\STATE $Q':= \roundfun{\tilde{Q}}{\delta_i}$; \COMMENT{rounding step}
\FOR {$\alpha, \beta \in B$}
\STATE $\eta(\alpha+\beta) := \# \{(\alpha',\beta') \in B^2 \mid \alpha'+\beta' = \alpha + \beta \}$;
\STATE $Q(\alpha,\beta) := Q'(\alpha,\beta) - \frac{1}{\eta(\alpha+\beta)} \Bigl( \sum_{\alpha'+\beta'=\alpha + \beta} Q'(\alpha',\beta')-g_{\frac{\alpha+\beta}{2}} \Bigr)$; \COMMENT{projection step}
\ENDFOR
\FOR {$k=1:l$}
\STATE $g_k:=\cup_{i=1}^{s_k}(2a^k_i\x^{\bu_i^k}+b_i^k\x^{\bv_i^k}-2c_i^k\x^{\bw_i^k})$;
\ENDFOR
\STATE \textbf{return} $\cup_{k=1}^lg_k$;
\end{algorithmic}
\end{algorithm}
}\fi

\subsection{Discussion on the quality of SONC lower bounds}\label{quality}
The quality of obtained SONC lower bounds depends on two successive steps: the relaxation to the corresponding PN-polynomial (from $\xi^*$ to $\xi_{sonc}$) and the relaxation to a specific simplex cover (from $\xi_{sonc}$ to $\xi_{socp}$). 
The loss of bound-quality at the second step can be improved by choosing a more optimal simplex cover. Nevertheless, it may happen that the loss of bound-quality at the first step is already big. Let us see an example.

\begin{example}
Let $f=1+x_1^4+x_2^4-x_1x_2^2-x_1^2x_2+5x_1x_2$. Since $\Lambda(f)$ forms a trellis, the simplex cover for $f$ is unique. One obtains  $\xi_{socp}=\xi_{sonc}\approx-6.916501$ while $\xi^*\approx-2.203372$. 
Hence the relative optimality gap is near $214\%$.
\end{example}

The above example indicates that the gap between nonnegative polynomials and SONC PN-polynomials (see figure \ref{fig-poly}) may greatly affect the quality of SONC lower bounds.

\begin{figure}[htbp]
\begin{center}
\begin{tikzpicture} 
\node at (-3.2,0) {{\footnotesize PSD polynomials}};
\node at (3.2,0) {{\footnotesize PN-polynomials}};
\node at (-0.6,0) {{\footnotesize SONC polynomials}};
\draw[color=blue] (-1.7,0) ellipse (3 and 2);
\draw[color=red] (1.7,0) ellipse (3 and 2);
\draw[color=green] (-0.6,0) circle (1.4);
\end{tikzpicture}
\end{center}
\caption{Relationship of different classes of polynomials}\label{fig-poly}
\end{figure}
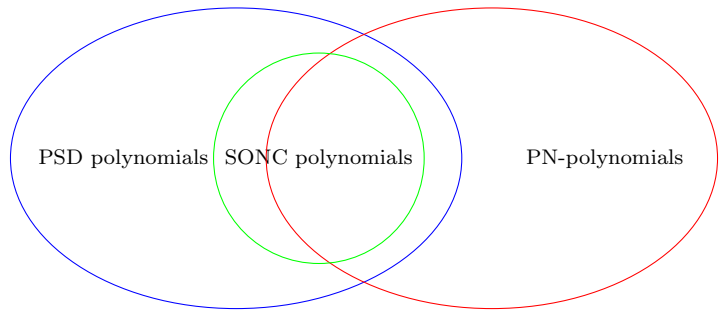

\section{Numerical experiments}
In this section, we present numerical results of the proposed algorithms for unconstrained polynomial optimization problems. 
Our tool, called {\tt SONCSOCP}, implements the simplex cover algorithm as well as the rational mediated set algorithm~\ref{alg3} ${\tt MedSet}$ and computes the optimal value $\xi_{socp}$ of the SOCP program~\eqref{sec3-eq7} with Mosek \cite{mosek}. 
All numerical experiments were performed on an Intel Core i5-8265U@1.60GHz CPU with 8GB RAM memory and the WINDOWS 10 system.
Our {\tt SONCSOCP} tool can be downloaded at  \href{https://github.com/wangjie212/SONCSOCP.}{github:{\tt SONCSOCP}}.

Our benchmarks are issued from the database of randomly generated polynomials provided by Seidler and de Wolff in \cite{se}. 
Depending on the Newton polytope, these benchmarks are divided into three classes: the ones with standard simplices, the ones with general simplices and the ones with arbitrary Newton polytopes (see \cite{se} for the details on the construction of these polynomials). We compare the performance of {\tt SONCSOCP} with the ones of {\tt POEM}, which relies on the ECOS solver to solve geometric programs (see \cite{se} for more details).

To measure the quality of a given lower bound $\xi_{lb}$, we rely on the `{\tt local\_min}' function available in {\tt POEM} which computes an upper bound $\xi_{min}$ on the minimum of a  polynomial. The relative optimality gap is defined by 
$\frac{|\xi_{min}-\xi_{lb}|}{|\xi_{min}|}.$

\begin{table}[htbp]
\caption{The notation}\label{table1}
\begin{center}
\begin{tabular}{|c|c|}
\hline
$n$&the number of variables\\
\hline
$d$&the degree\\
\hline
$t$&the number of terms\\
\hline
$l$& the lower bound on the number of inner terms\\
\hline
opt&the optimum\\
\hline
time&the running time in seconds\\
\hline
\end{tabular}
\end{center}
\end{table}

\subsection{Standard simplex}\label{standard}
For the standard simplex case, we take $10$ polynomials of different types (labeled by $N$). Running time and lower bounds obtained with {\tt SONCSOCP} and {\tt POEM} are displayed in Table \ref{tb2:standard}. Note that for polynomials with $\Lambda(\cdot)$ forming a trellis, the simplex cover is unique. Thus the SONC lower bounds obtained by {\tt SONCSOCP} and {\tt POEM} are the same theoretically, which is also reflected in Table \ref{tb2:standard}. 
For each polynomial, the relative optimality gap is less than $1\%$ and for $8$ out of $10$ polynomials, it is less than $0.1\%$ (see Figure \ref{fg2:standard}).

\begin{table}[htbp]
\caption{Results for the standard simplex case}\label{tb2:standard}
\begin{center}
\begin{tabular}{c|c|cccccccccc}
\multicolumn{2}{c|}{$N$}&$1$&$2$&$3$&$4$&$5$&$6$&$7$&$8$&$9$&$10$\\
\hline
\multicolumn{2}{c|}{$n$}&$10$&$10$&$10$&$20$&$20$&$20$&$30$&$30$&$40$&$40$\\
\multicolumn{2}{c|}{$d$}&$40$&$50$&$60$&$40$&$50$&$60$&$50$&$60$&$50$&$60$\\
\multicolumn{2}{c|}{$t$}&$20$&$20$&$20$&$30$&$30$&$30$&$50$&$50$&$100$&$100$\\
\hline
\multirow{2}*{time}&{\tt SONCSOCP}&$0.04$&$0.04$&$0.04$&$0.14$&$0.14$&$ 0.13$&$0.43$&$0.40$&$2.23$&$2.21$\\
\cline{3-12}
&{\tt POEM}&$0.26$&$0.27$&$0.26$&$0.43$&$0.44$&$0.42$&$1.78$&$1.79$&$2.20$&$2.25$\\
\hline
\multirow{2}*{opt}&{\tt SONCSOCP}&$3.52$&$3.52$&$3.52$&$2.64$&$2.64$&$2.64$&$2.94$&$2.94$&$4.41$&$4.41$\\
\cline{3-12}
&{\tt POEM}&$3.52$&$3.52$&$3.52$&$2.64$&$2.64$&$2.64$&$2.94$&$2.94$&$4.41$&$4.41$\\
\hline
\end{tabular}
\end{center}
\end{table}

\begin{figure}[htbp]
\begin{center}
\begin{tikzpicture}
\begin{axis}[xlabel={$N$},
ylabel={running time (s)},
legend pos=north west]
\addplot[color=blue,mark=o]
coordinates{(1,0.04)(2,0.04)(3,0.04)(4,0.14)(5,0.14)(6,0.13)(7,0.43)(8,0.40)(9,2.23)(10,2.21)};
\addlegendentry{SONCSOCP}
\addplot[color=red,mark=star]
coordinates{(1,0.26)(2,0.27)(3,0.26)(4,0.43)(5,0.44)(6,0.42)(7,1.78)(8,1.79)(9,2.20)(10,2.25)};
\addlegendentry{POEM}
\end{axis}
\end{tikzpicture}
\end{center}
\caption{Running time for the standard simplex case}\label{fg1:standard}
\end{figure}
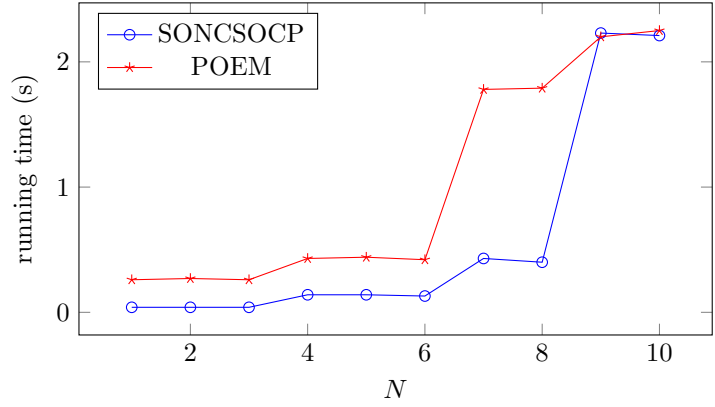

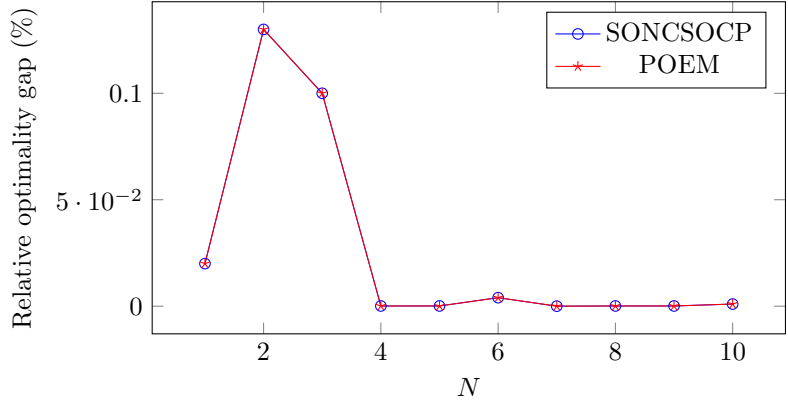
\begin{figure}[htbp]
\begin{center}
\begin{tikzpicture}
\begin{axis}[xlabel={$N$},
ylabel={Relative optimality gap (\%)},
legend pos=north east]
\addplot[color=blue,mark=o]
coordinates{(1,0.02)(2,0.13)(3,0.10)(4,0.0001)(5,0.0001)(6,0.004)(7,0.00001)(8,0.0001)(9,0.0001)(10,0.001)};
\addlegendentry{SONCSOCP}
\addplot[color=red,mark=star]
coordinates{(1,0.02)(2,0.13)(3,0.10)(4,0.0001)(5,0.0001)(6,0.004)(7,0.00001)(8,0.0001)(9,0.0001)(10,0.001)};
\addlegendentry{POEM}
\end{axis}
\end{tikzpicture}
\end{center}
\caption{Relative optimality gap for the standard simplex case}\label{fg2:standard}
\end{figure}

\subsection{General simplex}
For the general simplex case, we take $10$ polynomials of different types (labeled by $N$). 
Running time and lower bounds obtained with {\tt SONCSOCP} and {\tt POEM} are displayed in Table \ref{tb1:simplex}. 
As in Sec.~\ref{standard}, the SONC lower bounds obtained by {\tt SONCSOCP} and {\tt POEM} are the same. For each polynomial except for the one corresponding to $N = 7$, the relative optimality gap is within $30\%$, and for $6$ out of $10$ polynomials, the gap is below $1\%$ (see Figure \ref{fg2:simplex}). {\tt POEM} fails to obtain a lower bound for the instance $N = 10$ by returning $-$Inf.

 Figure \ref{fg1:simplex} shows that, overall, the running times of {\tt SONCSOCP} and {\tt POEM} are close. 
 {\tt SONCSOCP} is faster than {\tt POEM} for the instance $N = 6$, possibly because better  performance are obtained when the degree is relatively low.

\begin{table}[htbp]
\caption{Results for the general simplex case}\label{tb1:simplex}
\begin{center}
\begin{tabular}{c|c|cccccccccc}
\multicolumn{2}{c|}{$N$}&$1$&$2$&$3$&$4$&$5$&$6$&$7$&$8$&$9$&$10$\\
\hline
\multicolumn{2}{c|}{$n$}&$10$&$10$&$10$&$10$&$10$&$10$&$10$&$10$&$10$&$10$\\
\multicolumn{2}{c|}{$d$}&$20$&$30$&$40$&$50$&$60$&$20$&$30$&$40$&$50$&$60$\\
\multicolumn{2}{c|}{$t$}&$20$&$20$&$20$&$20$&$20$&$30$&$30$&$30$&$30$&$30$\\
\hline
\multirow{2}*{time}&{\tt SONCSOCP}&$0.32$&$0.29$&$0.36$&$0.48$&$0.54$&$ 0.56$&$0.73$&$0.88$&$1.04$&$1.04$\\
\cline{2-12}
&{\tt POEM}&$0.28$&$0.31$&$0.31$&$0.31$&$0.43$&$0.74$&$0.75$&$0.74$&$0.72$&$0.76$\\
\hline
\multirow{2}*{opt}&{\tt SONCSOCP}&$1.18$&$0.22$&$0.38$&$0.90$&$0.06$&$4.00$&$-4.64$&$1.62$&$2.95$&$5.40$\\
\cline{2-12}
&{\tt POEM}&$1.18$&$0.22$&$0.38$&$0.90$&$0.06$&$4.00$&$-4.64$&$1.62$&$2.95$&$-$Inf\\
\end{tabular}
\end{center}
\end{table}

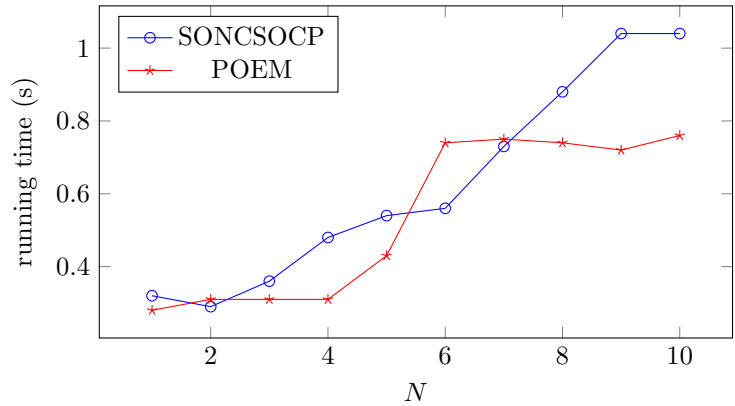
\begin{figure}[htbp]
\begin{center}
\begin{tikzpicture}
\begin{axis}[xlabel={$N$},
ylabel={running time (s)},
legend pos=north west]
\addplot[color=blue,mark=o]
coordinates{(1,0.32)(2,0.29)(3,0.36)(4,0.48)(5,0.54)(6,0.56)(7,0.73)(8,0.88)(9,1.04)(10,1.04)};
\addlegendentry{SONCSOCP}
\addplot[color=red,mark=star]
coordinates{(1,0.28)(2,0.31)(3,0.31)(4,0.31)(5,0.43)(6,0.74)(7,0.75)(8,0.74)(9,0.72)(10,0.76)};
\addlegendentry{POEM}
\end{axis}
\end{tikzpicture}
\end{center}
\caption{Running time for the general simplex case}\label{fg1:simplex}
\end{figure}

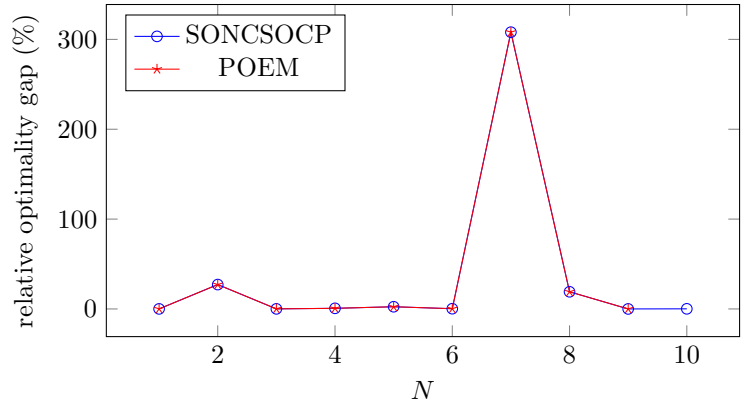
\begin{figure}[htbp]
\begin{center}
\begin{tikzpicture}
\begin{axis}[xlabel={$N$},
ylabel={relative optimality gap (\%)},
legend pos=north west]
\addplot[color=blue,mark=o]
coordinates{(1,0.069)(2,27.13)(3,0.102)(4,0.708)(5,2.46)(6,0.27)(7,308)(8,19.11)(9,0.016)(10,0.143)};
\addlegendentry{SONCSOCP}
\addplot[color=red,mark=star]
coordinates{(1,0.069)(2,27.13)(3,0.102)(4,0.708)(5,2.46)(6,0.27)(7,308)(8,19.11)(9,0.016)};
\addlegendentry{POEM}
\end{axis}
\end{tikzpicture}
\end{center}
\caption{Relative optimality gap for the general simplex case}\label{fg2:simplex}
\end{figure}

\subsection{Arbitrary polytope}
For the arbitrary polytope case, we take $20$ polynomials of different types (labeled by $N$). 
With regard to these examples, {\tt POEM} always throws an error ``expected square matrix''. 
Running time and lower bounds obtained with {\tt SONCSOCP} are displayed in Table \ref{tb1:arbitrary}. For each polynomial, the relative optimality gap is within $25\%$ and for $17$ out of $20$ polynomials, the gap is within $1\%$ (see Figure \ref{fg2:arbitrary}).

\begin{table}[htbp]
\caption{Results for the arbitrary polytope case}\label{tb1:arbitrary}
\begin{center}
\begin{tabular}{c|c|cccccccccc}
\multicolumn{2}{c|}{$N$}&$1$&$2$&$3$&$4$&$5$&$6$&$7$&$8$&$9$&$10$\\
\hline
\multicolumn{2}{c|}{$n$}&$10$&$10$&$10$&$10$&$10$&$10$&$10$&$10$&$10$&$10$\\
\multicolumn{2}{c|}{$d$}&$20$&$20$&$20$&$30$&$30$&$30$&$40$&$40$&$40$&$50$\\
\multicolumn{2}{c|}{$t$}&$30$&$100$&$300$&$30$&$100$&$300$&$30$&$100$&$300$&$30$\\
\multicolumn{2}{c|}{$l$}&$15$&$71$&$231$&$15$&$71$&$231$&$15$&$71$&$231$&$15$\\
\hline
\multirow{2}*{{\tt SONCSOCP}}&time&$0.38$&$1.75$&$6.86$&$0.64$&$3.13$&$ 11.3$&$0.72$&$4.01$&$14.6$&$0.76$\\
\cline{2-12}
&opt&$0.70$&$3.32$&$31.7$&$3.31$&$15.3$&$3.31$&$0.47$&$5.42$&$38.7$&$1.56$\\
\hline
\hline
\multicolumn{2}{c|}{$N$}&$11$&$12$&$13$&$14$&$15$&$16$&$17$&$18$&$19$&$20$\\
\hline
\multicolumn{2}{c|}{$n$}&$10$&$10$&$10$&$10$&$10$&$20$&$20$&$20$&$20$&$20$\\
\multicolumn{2}{c|}{$d$}&$50$&$50$&$60$&$60$&$60$&$30$&$30$&$40$&$40$&$40$\\
\multicolumn{2}{c|}{$t$}&$100$&$300$&$30$&$100$&$300$&$50$&$100$&$50$&$100$&$200$\\
\multicolumn{2}{c|}{$l$}&$71$&$231$&$15$&$71$&$231$&$5$&$15$&$5$&$15$&$35$\\
\hline
\multirow{2}*{{\tt SONCSOCP}}&time&$4.41$&$16.8$&$1.84$&$11.2$&$42.4$&$ 3.20$&$8.84$&$2.60$&$10.5$&$38.7$\\
\cline{2-12}
&opt&$0.20$&$7.00$&$3.31$&$2.52$&$23.4$&$0.70$&$4.91$&$4.13$&$2.81$&$9.97$\\
\hline
\end{tabular}
\end{center}
\end{table}

\begin{figure}[htbp]
\begin{center}
\begin{tikzpicture}
\begin{axis}[xlabel={$N$},
ylabel={relative optimality gap (\%)},
legend pos=north west]
\addplot[color=blue,mark=o]
coordinates{(1,1.56)(2,0.0015)(3,0.0006)(4,0.02)(5,0.004)(6,0.002)(7,0.25)(8,0.01)(9,0.0008)(10,0.026)(11,0.023)(12,0.008)(13,0.02)(14,0.07)(15,0.04)(16,0.01)(17,0.43)(18,0.0002)(19,22.9)(20,24.8)};
\addlegendentry{SONCSOCP}
\end{axis}
\end{tikzpicture}
\end{center}
\caption{Relative optimality gap for the arbitrary polytope case}\label{fg2:arbitrary}
\end{figure}
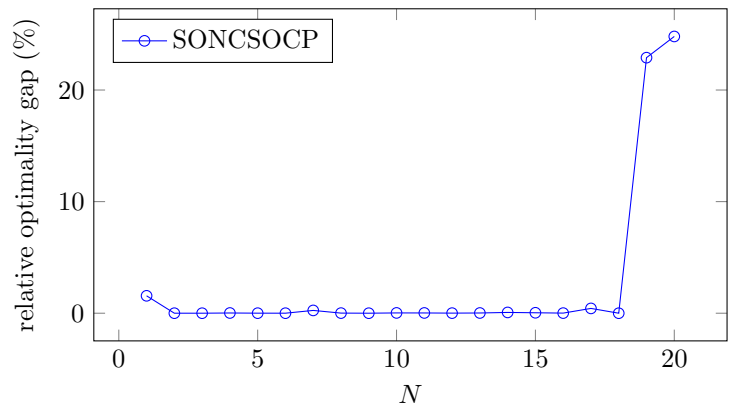

\section{Conclusions}
In this paper, we provide a constructive proof that each SONC cone admits a second-order cone representation. 
Based on this, we propose an algorithm to compute a lower bound for unconstrained polynomial optimization problems via second-order cone programming. 
Numerical experiments demonstrate the efficiency of our algorithm even when the number of variables and the degree are fairly large. 
Even though the complexity of our algorithm depends on the degree in theory, 
it turns out that this dependency is rather mild. 
It happens that for all numerical examples tested in this paper, the running time is below one minute even for polynomials of degree up to $60$.

Since the running time is satisfactory, the main concern of SONC-based algorithms for sparse polynomial optimization may be the quality of obtained lower bounds. 
For many examples tested in this paper, the relative optimality gap is within $1\%$. 
However, it can happen that the SONC lower bound is not accurate and this cannot be avoided by choosing an optimal simplex cover. 
In order to improve the quality of such bounds, it is mandatory to find more complex representations of nonnegative polynomials, which involve SONC polynomials. 
We leave it as a future work.
We also plan to design a rounding-projection procedure, in the spirit of \cite{pe}, to obtain exact nonnegativity certificates for polynomials lying in the interior of the SONC cone. 
A related investigation track is the complexity analysis and software implementation of the resulting hybrid numeric-symbolic scheme, as well as performance comparisons with concurrent methods based on semidefinite programming \cite{multivsos18} or geometric programming \cite{ma19}.

\bibliography{refer}

\appendix
\section{Algorithms}
\begin{algorithm}
\renewcommand{\algorithmicrequire}{\textbf{Input:}}
\renewcommand{\algorithmicensure}{\textbf{Output:}}
\caption{${\tt MedSeq}(p,q)$}\label{alg1}
\begin{algorithmic}[1]
\REQUIRE
$p,q\in\N,0<q<p$
\ENSURE
A sequence of tripes $\{(u_i,v_i,w_i)\}_i$ with $u_i=\frac{1}{2}(v_i+w_i)$ such that $\{0,p\}\cup\{u_i\}_i$ is a $(0,p)$-rational mediated sequence containing $q$
\STATE $u:=p$, $v:=q$, $w:=$gcd$(p,q)$;
\STATE $u:=\frac{u}{w}, v:=\frac{v}{w}$;
\IF{$2|u$}
  \IF{$v=\frac{u}{2}$}
  \STATE $A:=\{(1,0,2)\}$;
  \ELSE
  \IF{$v<u/2$}
    \STATE $A:={\tt MedSeq}(\frac{u}{2},v)\cup\{(\frac{u}{2},0,u)\}$;
    \ELSE
    \STATE $A:=\{(\frac{u}{2},0,u)\}\cup({\tt MedSeq}(\frac{u}{2},v-\frac{u}{2})+\frac{u}{2})$;
  \ENDIF
  \ENDIF
\ELSE
  \IF{$2|v$}
  \STATE Let $k,r\in\N^*$ such that $v=2^kr$ and $2\nmid r$;
    \IF{$v=u-r$}
    \STATE $A:=\{(\frac{1}{2}v,0,v),(\frac{3}{4}v,\frac{1}{2}v,v),\ldots,(v,v-r,u)\}$;
    \ELSE
      \IF{$v<u-r$}
      \STATE $A:=\{(\frac{1}{2}v,0,v),\ldots,(\frac{v-r+u}{2},v-r,u)\}\cup({\tt MedSeq}(\frac{u+r-v}{2},r)+v-r)$;
      \ELSE
      \STATE $A:=\{(\frac{1}{2}v,0,v),\ldots,(\frac{v-r+u}{2},v-r,u)\}\cup({\tt MedSeq}(\frac{u+r-v}{2},\frac{v+r-u}{2})+\frac{v+u-r}{2})$;
      \ENDIF
    \ENDIF
  \ELSE
  \STATE $A:=u-{\tt MedSeq}(u,u-v)$;
  \ENDIF
\ENDIF
\STATE \textbf{return} $wA$;
\end{algorithmic}
\end{algorithm}

\begin{algorithm}
\renewcommand{\algorithmicrequire}{\textbf{Input:}}
\renewcommand{\algorithmicensure}{\textbf{Output:}}
\caption{${\tt LMedSet}(\a_1,\a_2,\b)$}\label{alg2}
\begin{algorithmic}[1]
\REQUIRE
$\a_1,\a_2,\b\in\Q^n$ such that $\b$ lies on the line segment between $\a_1$ and $\a_2$
\ENSURE
A sequence of tripes $\{(\bu_i,\bv_i,\bw_i)\}_i$ with $\bu_i=\frac{1}{2}(\bv_i+\bw_i)$ such that $\{\a_1,\a_2\}\cup\{\bu_i\}_i$ is a $\{\a_1,\a_2\}$-rational mediated set containing $\b$
\STATE Let $\b=(1-\frac{q}{p})\a_1+\frac{q}{p}\a_2$, $p,q\in\N,0<q<p$,gcd$(p,q) = 1$;
\STATE $A:={\tt MedSeq}(p,q)$;
\STATE $M:=\cup_{(u,v,w)\in A}\{((1-\frac{u}{p})\a_1+\frac{u}{p}\a_2,(1-\frac{v}{p})\a_1+\frac{v}{p}\a_2,(1-\frac{w}{p})\a_1+\frac{w}{p}\a_2)\}$;
\STATE \textbf{return} $M$;
\end{algorithmic}
\end{algorithm}

\begin{algorithm}
\renewcommand{\algorithmicrequire}{\textbf{Input:}}
\renewcommand{\algorithmicensure}{\textbf{Output:}}
\caption{${\tt MedSet}(\A,\b)$}\label{alg3}
\begin{algorithmic}[1]
\REQUIRE
A trellis $\A=\{\a_1,\ldots,\a_m\}$ and a lattice point $\b\in\Conv(\A)^{\circ}$
\ENSURE
A sequence of tripes $\{(\bu_i,\bv_i,\bw_i)\}_i$ with $\bu_i=\frac{1}{2}(\bv_i+\bw_i)$ such that $\A\cup\{\bu_i\}_i$ is an $\A$-rational mediated set containing $\b$
\STATE Let $\b=\sum_{i=1}^m\frac{q_i}{p}\a_i$, where $p=\sum_{i=1}^mq_i$, $p,q_i\in\N^*$, $(p,q_1,\ldots,q_m)=1$;
\STATE $k:=1, \b_0:=\b$;
\WHILE{$k<m-1$}
\STATE $\b_k:=\frac{q_{k+1}}{p-(q_1+\cdots+q_k)}\a_{k+1}+\cdots+\frac{q_{m}}{p-(q_1+\cdots+q_k)}\a_{m}$;
\STATE $M_{k-1}:={\tt LMedSeq}(\a_{k},\b_{k},\b_{k-1})$;
\ENDWHILE
\STATE $M_{m-2}:={\tt LMedSeq}(\a_{m-1},\a_{m},\b_{m-2})$;
\STATE $M:=\cup_{i=0}^{m-2}M_i$;
\STATE \textbf{return} $M$;
\end{algorithmic}
\end{algorithm}

\begin{algorithm}
\renewcommand{\algorithmicrequire}{\textbf{Input:}}
\renewcommand{\algorithmicensure}{\textbf{Output:}}
\caption{{\tt SimplexCover($\Lambda(f), \Gamma(f)$)}}\label{alg0}
\begin{algorithmic}[1]
\REQUIRE
$\Lambda(f), \Gamma(f)$
\ENSURE
$\{(\A_k,\b_k)\}_k$: a set of pairs such that $\A_k\subseteq\Lambda(f)$ is a trellis and $\b_k\in\Gamma(f)\cap\Conv(\A_k)^{\circ}$
\STATE $U:=\Lambda(f), V:=\Gamma(f)$, $k:=0$;
\WHILE{$U\ne\emptyset$ and $V\ne\emptyset$}
\STATE $k:=k+1$;
\STATE Choose $\a_0\in U$ and $\b_k\in V$;
\STATE $\boldsymbol{\lambda}:=\textrm{SimSel}(\b_k,\Lambda(f),\a_0)$;
\STATE $\A_k:=\{\a\in\Lambda(f)\mid\lambda_{\a}>0\}$;
\STATE $U:=U\backslash\A_k$, $V:=V\backslash\{\b_k\}$;
\ENDWHILE
\IF{$V\ne\emptyset$}
\WHILE{$V\ne\emptyset$}
\IF{$U=\emptyset$}
\STATE $U:=\Lambda(f)$;
\ENDIF
\STATE $k:=k+1$;
\STATE Choose $\a_0\in U$ and $\b_k\in V$;
\STATE $\boldsymbol{\lambda}:=\textrm{SimSel}(\b_k,\Lambda(f),\a_0)$;
\STATE $\A_k:=\{\a\in\Lambda(f)\mid\lambda_{\a}>0\}$;
\STATE $U:=U\backslash\A_k$, $V:=V\backslash\{\b_k\}$;
\ENDWHILE
\ELSE
\WHILE{$U\ne\emptyset$}
\IF{$V=\emptyset$}
\STATE $V:=\Gamma(f)$;
\ENDIF
\STATE $k:=k+1$;
\STATE Choose $\a_0\in U$ and $\b_k\in V$;
\STATE $\boldsymbol{\lambda}:=\textrm{SimSel}(\b_k,\Lambda(f),\a_0)$;
\STATE $\A_k:=\{\a\in\Lambda(f)\mid\lambda_{\a}>0\}$;
\STATE $U:=U\backslash\A_k$, $V:=V\backslash\{\b_k\}$;
\ENDWHILE
\ENDIF
\STATE \textbf{return} $\{(\A_k,\b_k)\}_k$
\end{algorithmic}
\end{algorithm}

\end{document}